\documentclass[smallextended,envcountsect,]{svjour3}
\usepackage{amsmath,amssymb}
\usepackage{cite}
\usepackage[normalem]{ulem}
\usepackage{color}
\usepackage{hyperref}
\hypersetup{
    colorlinks=true,
    linkcolor=blue,     citecolor=red,     urlcolor=blue  }
\usepackage{mathptmx}   
\usepackage[weather,alpine,misc,geometry]{ifsym}
\usepackage{amsfonts}
\usepackage{graphicx}
\usepackage[normalem]{ulem}
\usepackage{cancel}

\newcommand{\al}{\alpha}
\newcommand{\be}{\beta}
\newcommand{\ga}{\gamma}

\newcommand{\de}{\delta}
\newcommand{\eps}{\varepsilon}
\newcommand{\bx}{\bar x}

\newcommand {\R} {\mathbb R}
\newcommand {\N} {\mathbb N}

\newcommand {\B} {\mathbb B}

\newcommand {\dom} {{\rm dom}\,}

\newcommand {\cl} {{\rm cl}\,}

\newcommand {\sd} {\partial}
\newcommand {\Int} {{\rm int}\,}

\newcommand{\ds}{\displaystyle}

\def\nbh{neighbourhood}
\def\es{\emptyset}
\def\lsc{lower semicontinuous}

\def\LHS{left-hand side}
\def\RHS{right-hand side}

\def\EVP{Ekeland variational principle}
\def\Fr{Fr\'echet}

\newcommand{\norm}[1]{\left\Vert#1\right\Vert}

\newcommand {\diam} {{\rm diam}\,}

\newcommand{\red}[1]{\textcolor{red}{#1}}

\newcommand{\olive}[1]{\textcolor{olive}{#1}}

\newcommand{\ang}[1]{\left\langle #1 \right\rangle}
\newcommand{\qdtx}[1]{\quad\mbox{#1}\quad}
\newcommand{\AND}{\quad\mbox{and}\quad}
\newcounter{mycount}

\newcommand{\AK}[1]{\todo[inline]{AK {#1}}}
\newcommand{\AH}[1]{\todo[inline,color=green!40]{AH {#1}}}


\smartqed
\usepackage{enumitem}
\setlist[enumerate,1]{label={\rm(\roman*)}}
\setlist[enumerate,2]{label={\rm(\alph*)}}

\newcommand{\PM}[1]{\todo[inline,color=blue!40]{Reviewer: {#1}}}

\newcommand{\sdc}{{\partial}^{C}}
\newcommand{\sdf}{\partial}

\begin{document}


\title{Optimality Conditions and Subdifferential Calculus for Infinite Sums of Functions}

\author{Abderrahim Hantoute
\and Alexander Y. Kruger
\and Marco A. L\'opez}

\institute{
Abderrahim Hantoute \at
Department of Mathematics, University of Alicante, Spain\\
\email{hantoute@ua.es}, ORCID: 0000-0002-7347-048X
\and
Alexander Y. Kruger (corresponding author) \at
Optimization Research Group,
Faculty of Mathematics and Statistics,
Ton Duc Thang University, Ho Chi Minh City, Vietnam\\ \email{alexanderkruger@tdtu.edu.vn},
ORCID: 0000-0002-7861-7380
\and
Marco A. L\'opez \at
Department of Mathematics, University of Alicante, Spain\\
\email{marco.antonio@ua.es},
ORCID: 0000-0002-0619-9618}
\maketitle

\date{}
\if{
\AK{17/06/25.
1. Example 1.1 (uniform lower semicontinuity vs firm uniform lower semicontinuity) and Proposition 2.1 (firm uniform lower semicontinuity of indicator functions) added.

2. ``Notation and preliminaries'' included in the Introduction.

3. Example 2.1 (by Abderrahim) reworked.

4. Remark 2.5 (i) slightly expanded.

5. Quasiuniform lower semicontinuity (the last section) is now ``on a set''. Multiple changes.

6. References fixed. Thanks, Abderrahim.
}
\AK{15/06/25.
1. Introduction expanded using Abderrahim's suggestions from 4/06/25.

2. Weak uniform lower semicontinuity and weak firm uniform lower semicontinuity as well as their ``quasi'' analogues eliminated as discussed 13/06/25.

3. The statements of Theorems \ref{T6.01} and \ref{T5.2} shortened.
The second part of the former Theorem~\ref{T5.2} is now Proposition 5.7.
I have failed to move it to Sect. 3 as its assumptions guarantee quasiuniform lower semicontinuity.
\AK{12/06/25.
Multiple changes in response to Abderrahim's comments, especially in Theorems~\ref{T6.1} and \ref{T7.1} and their proofs.}

\AK{6/06/25.
1. Abstract updated.

2. Some ``motivation'' stuff added to the Introduction.

3. Description of the content of Section~\ref{S3} expanded.

4. Description of the content of the ``quasi'' Section~\ref{quasi} added.

5. A portion of the former Introduction containing formulas moved to Section~\ref{S3}.

6. Conclusions expanded.
In particular, a short description of the forthcoming paper from NotesMarco31May added.

7. Some minor editing here and there.
}

\AK{4/06/25.
1. Quasi properties moved from Sect.~\ref{S3} to the new Sect.~\ref{quasi} after the optimality conditions.
The main Theorems~\ref{T6.1} and \ref{T7.1} are slightly simplified.
Their original versions are given with some minor hints of the proofs in Sect.~\ref{quasi}.

2. Former Section~5 deleted.
}

\AK{2/06/25.
Uniform lower semicontinuity properties on $U$ replaced in most cases by the properties on the whole (metric) space, and subscript $U$ removed from the basic notations.
There are, nevertheless, instances where we need properties on subsets.}

\AK{30/05/25.
Most of the red colour removed (assuming that those old changes have already been accepted by everybody) in preparation for more dramatic changes.
Some minor editing here and there.}

\AK{18/05/25.
1. Ref. [2] added (\EVP); ref [8] (formerly [12]) updated.

2. ``Robustness'' replaced by ``$\inf$-stability''.

3. Minor changes in the last section reflecting our discussion on 15/05/25.
}

\AK{12/05/25.
1. Springer style applied.

2. Abstract, Introduction and Conclusions updated to comply with the recent changes in the main part.

3. Some definitions involving formulas moved from the Introduction to Section~\ref{S3}.
They have been slightly shortened and reorganised.

4. I have tried to address some remarks of the reviewer concerning (former) Section~5.
}\fi

\begin{abstract}
The paper extends the widely used in optimisation theory decoupling techniques to infinite collections of functions.
Extended concepts of uniform lower semicontinuity and firm uniform lower semicontinuity
are discussed.
The main theorems give fuzzy subdifferential necessary conditions (multiplier rules) for a local minimum of the sum of an infinite collection of functions and fuzzy subdifferential sum rules
without the traditional Lipschitz continuity assumptions.
More subtle ``quasi'' versions of the uniform infimum and uniform lower semicontinuity properties are also discussed.

\end{abstract}




\section{Introduction}\label{sec:introduction}

The main idea of the \emph{decoupling approach} to developing subdifferential rules in nonsmooth optimisation and calculus consists in allowing the functions involved in the problem to depend on their own variables while ensuring that the distance between the variables tends to zero.
This allows one to express the resulting conditions in terms of subdifferentials of the individual functions and/or normal cones to individual sets, or appropriate primal space tools.
The approach has been intuitively used in numerous publications for decades.
As demonstrated in \cite{Las01} and summarized in \cite[Section~6.1.4]{BorZhu05}, all basic subdifferential rules in Banach spaces are
different facets of a variational principle in conjunction with a decoupling method.

The basics of the decoupling approach were formalized at the end of the 1990s -- beginning of the 2000s in Borwein and Ioffe \cite{BorIof96}, Borwein and Zhu \cite{BorZhu96} and Lassonde \cite{Las01} (see also \cite{BorZhu05,Pen13,FabKruMeh24}): several useful
concepts, describing the joint behaviour of finite collections of functions typical for optimisation problems, were introduced and investigated.
Employing the
concepts and techniques
developed in \cite{BorZhu96,BorIof96,Las01} allowed researchers to streamline proofs of optimality conditions and calculus relations, unify and simplify the respective statements, as well as clarify and in many cases weaken the assumptions.
In particular, the widely spread (even now) belief that the traditional assumption of Lipschitz continuity of all but one function in
subdifferential
rules is absolutely necessary was demonstrated to be false.

The decoupling techniques from \cite{BorZhu96,BorIof96,Las01} are applicable to a finite number of functions.
Let $T$ be a nonempty finite index set.
Given a collection of functions $\{f_{t}\}_{t\in T}$ from a metric space $X$ to $\R_{\infty}:=\R\cup\{+\infty\}$, one can define the \emph{uniform infimum}
of their sum:
\begin{align}
\label{La0-}
{\Lambda}(\{f_t\}_{t\in T}) :=\liminf\limits_{\substack{\diam\{x_t\}_{t\in T}\to0}}\; \sum_{t\in T}f_{t}(x_{t}).
\end{align}
\if{
\red{************}
\begin{align*}
{\Lambda}(\{f_t\}_{t\in T}) :=&\sup_{\de>0}\inf_{\substack{\diam\{x_t\}_{t\in T}<\de}}\; \sum_{t\in T}f_{t}(x_{t})
\\
=&\sup_{\de>0}\;\inf_{\substack{\|x_t-x\|<\de,\; x,x_t\in X\; (t\in T)}}\; \sum_{t\in T}f_{t}(x_{t})
\\
=&\sup_{\de>0}\;\inf_{\substack{x,x_t\in X\; (t\in T)}}\; \sum_{t\in T}(f_{t}+i_{B_\de(x)})(x_{t})
\\
=&\sup_{\de>0}\;\inf_{x\in X}\; \sum_{t\in T}\inf_{x_t\in X}(f_{t}+i_{B_\de(x)})(x_{t})
\\
=&\sup_{\de>0}\;\inf_{\substack{x\in X}}\; \sum_{t\in T}(f_{t}\square i_{\de\B})(x).
\\
(f\square g)(x):=&\inf_u(f(u)+g(x-u))
\end{align*}

\red{************}
}\fi
Here, $\{x_t\}_{t\in T}$ denotes a collection of points $x_t$ $(t\in T)$, while $\diam Q:=\sup_{x_1,x_2\in Q}d(x_1,x_2)$ is the diameter of a set $Q$.
Definition \eqref{La0-} replaces minimizing
the conventional sum $\sum_{t\in T}f_{t}$ by minimizing
the \emph{decoupled sum} $\{x_t\}_{t\in T}\mapsto\sum_{t\in T}f_{t}(x_{t})$
{with $\diam\{x_t\}_{t\in T}\to0$}.
One obviously has $\Lambda(\{f_t\}_{t\in T})\le \inf\sum_{t\in T}f_{t}$, and the inequality can be strict;
see Example~\ref{E2.1}.
\sloppy

\if{
Definition \eqref{La0-} can be rewritten equivalently as
\begin{align*}
{\Lambda}(\{f_t\}_{t\in T}) ={\sup_{\de>0}}\;\inf_{\substack{x\in X}}\; \sum_{t\in T}(f_{t}\square i_{\de\B})(x).
\end{align*}
Here, $i_U$ is the \emph{indicator function} of a set $U$: $i_U(x)=0$ if $x\in U$ and $i_U(x)=+\infty$ if $x\notin U$, and $\square$ denotes the \emph{infimal convolution} of functions: if $f,g:X\to\R_\infty$, then $(f\square g)(x):=\inf_{u\in X}(f(u)+g(x-u))$ for all $x\in X$.
}\fi

The collection $\{f_{t}\}_{t\in T}$ is said to be \emph{uniformly lower semicontinuous}
(on $X$) if
\begin{gather}
\label{La0qc}
\inf\sum_{t\in T}f_{t}\le
\Lambda(\{f_t\}_{t\in T}).
\end{gather}
Inequality \eqref{La0qc}
(in view of
{the above remark},
it can only hold as equality) plays the role of a qualification condition.
It lies at the core of the decoupling approach.
Some
{typical}
sufficient conditions for \eqref{La0qc} can be found in \cite{Las01,Pen13,FabKruMeh24}.

If the infimum in the \LHS\ of \eqref{La0qc} is attained at some point $\bx\in X$, i.e.,
$\sum_{t\in T}f_{t}(\bx)\le
\Lambda(\{f_t\}_{t\in T})$,
the point $\bx$ is called a \emph{uniform minimum}  \cite{Las01}
of $\sum_{t\in T}f_{t}$.
Every uniform minimum is obviously a conventional minimum, while the converse implication is not true in general,
and this is where qualification conditions come into play.
{Their role is typically to ensure the converse implication, i.e., the validity of inequality \eqref{La0qc}.}

When $X$ is a subset of a normed vector space, a slightly different
variant of uniform infimum is used in \cite{Las01,BorZhu05,Pen13} along with \eqref{La0-}.
The corresponding quantity is in general smaller than \eqref{La0-}.
The differences between the two concepts of uniform infimum and the corresponding
concepts of uniform lower semicontinuity are not essential for the subject of this paper, and we will only discuss in what follows
{definition \eqref{La0-} and its extensions}.

{Condition \eqref{La0qc} compares the optimal values of the conventional and decoupled minimization problems.
The corresponding minimizing sequences and minimum points (if exist) may lie far apart.}
A more advanced \emph{firm uniform lower semicontinuity} notion was introduced in \cite{FabKruMeh24} as an analytical counterpart of the
\emph{sequential uniform lower semicontinuity (ULC property)} used in \cite{BorIof96,BorZhu05}.
The collection $\{f_{t}\}_{t\in T}$ is firmly uniformly lower semicontinuous
(on $X$)
if ${\Theta}(\{f_t\}_{t\in T})=0$, where
\begin{align}
\label{Th}
{\Theta}(\{f_t\}_{t\in T}) :=\limsup_{\substack{\diam\{x_t\}_{t\in T}\to0\\ x_t\in\dom f_t\;(t\in T)}} \inf_{x\in X} \max\Big\{\max_{t\in T}d(x,x_t),\sum_{t\in T} (f_t(x)-f_{t}(x_{t}))\Big\}.
\end{align}
\if{
{\AH{08-08-2024. To be consistent with the convention $+\infty-(+\infty)=+\infty$ used in the last section, I would propose to use it here and remove the additional condition $x_t\in\dom f_t$ $(t\in T)$ from the definition of ${\Theta}_{U}(\{f_t\}_{t\in T})$. This change would not alter the value of this constant but would make it easier.}}
\AK{12/08/24.
I am afraid it would.
To preserve the value we would need the opposite convention: $+\infty-(+\infty)=-\infty$ (or $=0$).}
}\fi
Definition \eqref{Th}
combines minimizing
the conventional sum (over $x$) and the decoupled one (over $x_t$, $t\in T$)
as well as minimizing the distances $x\mapsto d(x,x_t)$ $(t\in T)$, thus, forcing the variables of the two mentioned minimization problems to be close.
\if{
\AH{Observe that the set $U$ can always be chosen smaller and smaller, so that the distance between the $x_t$'s could be made smaller}
\AK{12/08/24.
$U$ is a fixed set here; can be the whole space.}
}\fi
The maximum of the distances $d(x,x_t)$ over $t\in T$ can 
be replaced in \eqref{Th}
(thanks to the condition $\diam\{x_t\}_{t\in T}{\to0}$)
by a single distance $d(x,x_t)$ for some $t\in T$.
Compared to \eqref{La0-}, definition \eqref{Th}
contains
{additional}
technical conditions $x_t\in\dom f_t$ $(t\in T)$.
Their role is to exclude the undesirable indefinite form $+\infty-(+\infty)$ in the \RHS\ of \eqref{Th}.
Note that these conditions are also implicitly present in \eqref{La0-}, though their absence in the definition does not cause problems.
Condition ${\Theta}(\{f_t\}_{t\in T})=0$ can be strictly stronger than \eqref{La0qc}; see examples in \cite[Section~3]{FabKruMeh24}.


In the current paper, we extend the decoupling techniques developed in \cite{BorZhu96,BorIof96,Las01,FabKruMeh24} to arbitrary collections of functions.
\if{
\PM{
``in the overall introduction, not a single sentence
can be found which motivates the present study and the precise setting,
i.e., why should such concepts be studied (apart from academic interest).
This is not acceptable.

``in Section 5, a reduction argument is used to be able to restrict to
the classical setting $T :=\N$. If there is no application-driven need for more
generality or another convincing reason for it, $T :=\N$ should be used.''}
}\fi
Infinite collections of functions are ubiquitous in analysis and optimization.
The most well known examples come from the classical definition of integral as the limit of finite sums and Taylor series expansion of a function.
In the first case, if the integrand involves a parameter, natural questions are whether its differentiability properties with respect to the parameter are inherited by the integral, and whether the derivative (subdifferential) of the integral can be represented via the derivative (subdifferential) of the integrand.
\if{
\AK{6/06/25.
The integral model described in HLK8-1 does not seem to fall into the framework of this paper.}
{\AH{10/06/25. Yes; for that we should extend our model in future opportunity.}}
}\fi

The need to minimize the sum $\sum_{t\in T}f_{t}$ of a possibly infinite collection of functions arises rather often and naturally leads to minimizing partial sums $\sum_{t\in S}f_{t}$ corresponding to finite subsets $S\subset T$.
In concrete situations, $S$
corresponds to observed data samples.
Employing larger subsets $S$ means access to more data and, so, more accurate and traceable optimal
solutions. However, because solving all such subproblems individually may be
inconceivable, one approach consists of considering the worst-case scenario over all finite
subsets by appealing to the \emph{minimax robust} optimization: minimizing the supremum of
$\sum_{t\in S}f_{t}$ over all finite subsets $S\subset T$ (see, e.g.,
\cite{BotJeyLi13,DinGobVol19,DinGobVol20}).
This approach has found numerous application in stochastic optimization, particularly
within \emph{distributionally robust} optimization  frameworks (as detailed in \cite{XieWei24}).
However, the approach has limitations related to the
possible non-existence of or high cost of computing robust solutions,
often due to non-differentiability of the associated supremum function.
To overcome this difficulty, an alternative robust model of minimizing the (upper) infinite sum $\overline\sum_{t\in T}f_{t}$ can be
considered.
The latter is defined for all $x$ as the upper limit of $\sum_{t\in S}f_{t}(x)$ over all finite subsets $S\uparrow T$ (in the sense of ascending inclusions).
The objective function in this formulation is likely to
behave more smoothly than the supremum function mentioned above.
The latter
model can be more appropriate than the minimax robust problem as long as
one is interested in minimizing the entire sum of the $f_{t}$'s.
Besides, when $T$ is finite, it reduces to minimization of the
usual sum, while the minimax robust problem has a completely different
interpretation.
A related infinite sum was studied in \cite{LucVol19, HanJouVic23}
(see also \cite{
ZheNg04,LiNg08,ValZal16}),
provided that the above upper limit
exists as a limit. Subdifferential calculus rules were then established for
this operation in the context of convex analysis. The approach of \cite{LucVol19}
adopts some extensions of classical
techniques of integration to infinite collections of functions for the case of
countable sums. The development in \cite{HanJouVic23}
performs a reduction step that allows one to go from arbitrary infinite to countable sums, and then use the classical theory of normal convex integrals
\cite{
Iof06,LopThi08,HanJou18,CorHanPer19,CorHanPer21}.

Infinite optimization can benefit from our approach, as long as the problems are reformulated as  minimization of suitable infinite sums \cite{CorHanLop23, GobLop98}.
In the context of convex optimization, particularly convex duality theory, our infinite sum framework can be useful when constructing tractable dual problems that exhibit strong duality, i.e., problems with \emph{zero duality gap} and with dual solutions that are attained.
In contrast to the existing dual (e.g., of Haar type), fuzzy multiplier
rules and general optimality conditions can be produced under reasonable
sufficient conditions, namely of \emph{Slater} type due to the convex setting.
Several dual formulations have been proposed for such problems
(see, e.g., \cite{Duf56,Kre61,Bor83}). A
comprehensive treatment of the theory in the linear case can be found in \cite[Chapter~3]{AndNas87}.
In general infinite optimization, infinite sums offer a useful tool for defining
alternative dual problems with zero duality
gap and strong duality properties.
We refer the readers to 
\cite{HanKruLop2} for more details.

We extend definitions \eqref{La0-} and \eqref{Th} to the setting of an arbitrary index set $T$.
This leads to the corresponding extensions of the definitions of uniform lower semicontinuity and firm uniform lower semicontinuity.
The infinite sum of functions is defined as the upper limit of sums over the family of all finite subsets of $T$ considered as a directed set.
Several characterizations and sufficient conditions ensuring uniform and firm uniform lower semicontinuity are provided,
thus, extending the decoupling theory to infinite collections of functions.
We show that the firm uniform lower semicontinuity property is stable under uniformly continuous perturbations of one of the functions.
The new 
definitions and results are
discussed in detail in Section~\ref{S3}.

In Section~\ref{S6}, we apply the extended decoupling techniques to deriving
fuzzy subdifferential necessary conditions (multiplier rules) for a
local minimum of the sum of an arbitrary collection of functions.
In general Banach spaces, the conditions are formulated in terms of Clarke subdifferentials, while in Asplund spaces, \Fr\ subdifferentials are used.
As a consequence,
we establish
a generalized version of the (strong) fuzzy sum rule for Fr\'{e}chet subdifferentials of an arbitrary collection of functions without the traditional Lipschitz continuity assumptions.

In Section~\ref{quasi}, we discuss more subtle ``quasi'' versions of the uniform infimum and uniform lower semicontinuity properties studied in Section~\ref{S3}, which are still sufficient for the multiplier and sum rules in Section~\ref{S6}.
In the case of a pair of functions, the ``quasi'' properties were used in \cite{FabKruMeh24} when establishing general optimality conditions and subdifferential calculus formulas in non-Lipschitzian settings, strengthening some results from \cite{BorZhu96,BorIof96,Las01}.
Some ``quasi'' statements do not have analogues in Section~\ref{S3}.
For instance, we show that, if a family of functions is firmly quasiuniformly lower semicontinuous on a subset $U\subset X$, then it is firmly quasiuniformly lower semicontinuous on any subset of $U$.

Some conclusions are presented in Section~\ref{sec:conclusions}.
We introduce there our forthcoming paper \cite{HanKruLop2} dedicated to applications of the results in Sections~\ref{S6} and \ref{quasi} to optimality of an infinite sum of functions coming from semi(infinite) optimization.

\paragraph{Notation and preliminaries}
\label{S2}

Our basic notation is standard; cf., e.g.,
\cite{Kru03,Mor06.1,Iof17}.
Throughout the paper $X$ is either a metric or a normed (typically, Banach or Asplund) space.
Recall that a Banach space is \emph{Asplund} if every continuous convex function on an open convex set is Fr\'echet differentiable on a dense subset \cite{Phe93}, or equivalently, if the dual of each
separable subspace is separable.
All reflexive, particularly, all finite-dimensional Banach spaces are Asplund.
We refer the reader to
\cite{Phe93,BorZhu05,Mor06.1} for discussions about and characterizations of Asplund spaces.

We use the same notations
$d(\cdot,\cdot)$ and $\norm{\cdot}$
for distances
(including point-to-set
distances)
and norms in all spaces. 
Symbol $0$ denotes the zero vector in all linear spaces.
Normed spaces are often treated as metric spaces with the distances induced by the norms.
If $X$ is a normed space, its topological dual is denoted by $X^*$, while $\langle\cdot,\cdot\rangle$
denotes the bilinear form defining the pairing between the two spaces.
In a metric space, $\B$ and $\overline\B$ are the \emph{open} and \emph{closed} unit balls, while $B_\de(x)$ and $\overline B_\de(x)$ are the
\emph{open} and \emph{closed}
balls with radius $\de>0$ and centre $x$, respectively.
We write $\B^*$
to denote the \emph{open}
unit ball in the dual to a normed space.

Given a subset $U\subset X$,
the notations $\Int U$ and $\cl U$ represent the interior and closure of $U$, respectively.
The indicator function of $U$ is defined by $i_U(x)=0$ if $x\in U$ and $i_U(x)=+\infty$ if $x\notin U$.
In a normed space, $\cl^w$ denotes the closure in the weak topology.
Given a number $\de>0$, we set  $B_\de(U):=\bigcup_{x\in U}B_\de(x)$.

Symbols $\R$, $\R_+$ and $\N$ represent the sets of all real numbers, all nonnegative real numbers and all positive integers, respectively, and we denote $\R_{\infty}:=\R\cup\{+\infty\}$.
We make use of the conventions
$\inf\es_{\R}=+\infty$ and
$\sup\es_{\R_+}=0$, where $\es$ (possibly with a subscript) denotes the empty subset (of a given set).

For an extended-real-valued function $f\colon X\to\R_{\infty}$,
its domain is defined by
$\dom f:=\{x \in X\mid {f(x) < +\infty}\}$.
If $X$ is a normed space, and $\bar x\in\dom f$, the \emph{\Fr} and \emph{Clarke subdifferentials} of $f$ at $\bar x$ can be defined as
(see, e.g., \cite{Roc79,Cla83,Kru03,Mor06.1,Iof17})
\begin{gather*}
	\sdf f(\bar x)
	:=
	\left\{x^*\in X^*\,\middle|\,
		\liminf_{\bx\ne x\to\bar x}
			\frac{ f(x)- f(\bar x)-\langle x^*,x-\bar x\rangle}{\norm{x-\bar x}}
		\geq 0
	\right\},
\\
\sd^Cf(\bar x):=
	\left\{x^*\in X^*\mid\ang{x^*,u}\le f^\uparrow(\bx,u)
\qdtx{for all}
u\in X
	\right\},
\end{gather*}
where
$$
f^\uparrow(\bx,u):=\lim_{\rho\downarrow0}\; \sup_{\substack{(x,\mu)\in B_\rho(\bx,f(\bx))\\ t\in(0,\rho),\;f(x)\le\mu}}\;\inf_{z\in B_\rho(u)}\frac{f(x+tz)-\mu}t
$$
is the \emph{upper subderivative} \cite{Roc79} of $f$ at $\bar x$ in direction $u\in X$.
The latter expression reduces to the \emph{Clarke directional derivative} when $f$ is Lipschitz continuous near $\bx$.
If $\bx$ is a local minimum of $f$, then obviously $0\in\sdf f(\bx)$ (\emph{Fermat rule}).
It always holds $\sd f(\bx)\subset\sd^Cf(\bx)$, and both sets reduce to the conventional subdifferential of convex analysis when $f$ is convex:
\begin{gather*}
\sdf f(\bar x)=\sdc f(\bar x)=
\{x^*\in X^*\mid
f(x)-f(\bar x)-\langle x^*,x-\bar x\rangle\geq0
\qdtx{for all}
x\in X\}.
\end{gather*}

Next, we recall (a version of) the celebrated Ekeland variational principle and some standard subdifferential sum rules; see, e.g.,
\cite{AubFra90,BorZhu05,Mor06.1,Iof17}.

\begin{lemma}[\EVP]
\label{Ekeland}
Let $X$ be a complete metric space,
$f\colon X\to\R_{\infty}$ lower semicontinuous and bounded from below, and
$\bx\in\dom f$.
Then, for any $\varepsilon>0$, there exists an $\hat x\in X$ such that
$f(\hat x)\le  f(\bx)$ and
$f(\hat x)< f(x)+\varepsilon d(x,\hat x)$ for all
$x\in X\setminus\{\hat x\}$.
\end{lemma}

\begin{lemma}
[Subdifferential sum rules]
\label{SR}
Let $X$ be a Banach space,
$f_1,f_2\colon X\to\R_{\infty}$,
and $\bar x\in\dom f_1\cap\dom f_2$.
\begin{enumerate}
\item
\label{SR.1}
{\rm Convex rule}:
If $f_1$ and $f_2$ are convex, and $f_1$ is continuous at a point in $\dom f_2$, then
${\sd}(f_1+f_2)(\bar x)={\sd} f_1(\bar x)+{\sd} f_2(\bar x).$
\item
\label{SR.Clarke}
{\rm Clarke rule}:
If $f_1$ is Lipschitz continuous near $\bar x$,
and $f_2$ is lower semicontinuous near $\bar x$,
then
$\sdc(f_1+f_2)(\bar x)\subset \sdc f_1(\bar x)+\sdc f_2(\bar x).$
\item
\label{SR.2}
{\rm \Fr\ rule}:
If $X$ is Asplund, $f_1$ is Lipschitz continuous near $\bar x$,
and $f_2$ is lower semicontinuous near $\bar x$,
then, for any $x^*\in{\sdf}(f_1+f_2)(\bar x)$ and $\varepsilon>0$,
there exist points $x_1,x_2\in X$
such that $\norm{x_i-\bar x}<\varepsilon$, $|f_i(x_i)-f_i(\bx)|<\varepsilon$
$(i=1,2)$,
and
$x^*\in{\sdf}f_1(x_1)+{\sdf}f_2(x_2)+\varepsilon\B^*$.
\end{enumerate}
\end{lemma}

\section{Uniform lower semicontinuity}
\label{S3}

\if{
\AK{24/06/23.
Uniform or robust?
I think, `robust' in this context comes from Penot \cite{Pen13}.
It was not used in either the original publications \cite{BorIof96,BorZhu96,Las01} or the subsequent book \cite{BorZhu05}.
In \cite{FabKruMeh24}, we use `uniform'.}
}\fi

From now on, $T$ denotes an arbitrary nonempty index set, which can be infinite.
In this section, we study \emph{uniform} and \emph{firm uniform lower semicontinuity} of a collection of functions $\{f_{t}\}_{t\in T}$ on a metric space $X$.
These properties are crucial for the subdifferential necessary conditions (multiplier rules) and subdifferential calculus in the next section.
The section is a little technical.
It consists of a sequence of simple definitions and assertions building on the corresponding facts from \cite{BorZhu96,BorIof96,Las01,FabKruMeh24}.
The assertions are accompanied by very short proofs whenever necessary.
\if{
\PM{
``I strongly recommend to present the overall paper in the context of Banach
spaces. This alone reduces the technicality of the paper enormously.''}
}\fi

The definitions below involve the family of all finite subsets of indices
\begin{gather}
\label{F(T)}
\mathcal{F}(T):=\{S\subset T\mid {|S|<\infty}\},
\end{gather}
where $|S|$ denotes the \emph{cardinality} (number of elements) of the set $S$.
This family is considered as a \emph{directed set} endowed with the partial order determined by ascending inclusions.
Given a collection
(of extended real numbers) $\{\alpha_{S}\}_{S\in\mathcal{F}(T)} \subset[-\infty,+\infty],$ we use the notations
\begin{align}
\label{al}
\limsup_{S\uparrow T,\;|S|<\infty}\al_S
:=\inf_{S_0\in\mathcal{F}(T)} \sup_{S\in\mathcal{F}(T),\;S_0\subset S}\al_S,\;\;
\liminf_{S\uparrow T,\;|S|<\infty}\al_S
:=\sup_{S_0\in\mathcal{F}(T)} \inf_{\substack{S\in\mathcal{F}(T),\;S_0\subset S}}\al_S
\end{align}
for the, respectively, upper and lower limits of $\{\alpha_{S}\}_{S\in\mathcal{F}(T)}$, and write $\lim_{S\uparrow T}\al_S$ if the
limits are equal.
This is the case, for instance, when
$T$ is finite (hence, $\lim_{S\uparrow T}\al_S=\al_T$) or the mapping $\mathcal{F}(T)\ni S\mapsto\alpha_{S}$
is monotone.
Recall that $S\mapsto\alpha_{S}$ is non-decreasing (non-increasing) if $\alpha_{S_1}\ge\alpha_{S_2}$ ($\alpha_{S_1}\le\alpha_{S_2}$) for all $S_1,S_2\in\mathcal{F}(T)$ with $S_1\subset S_2$.
The mappings $S_0\mapsto\sup_{S\in\mathcal{F}(T),\;S_0\subset S}\al_S$ and $S_0\mapsto\inf_{S\in\mathcal{F}(T),\;S_0\subset S}\al_S$ are non-increasing and non-decreasing, respectively.
Hence, the limits \eqref{al} are well defined, and $\limsup_{S\uparrow T,\;|S|<\infty}\al_S\ge\liminf_{S\uparrow T,\;|S|<\infty}\al_S$.
Given a sequence of real numbers $t_1,t_2,\ldots$, definitions \eqref{al} reduce to the conventional upper and lower limits, respectively, if we set $\al_S:=t_{\max S}$ for any finite subset $S\subset\N$.
\if{
\AK{6/06/24.
Is this terminology acceptable?
I am using it to get rid of $S_0$ in the main statements.}
}\fi

Given a collection of functions $\{f_{t}\}_{t\in T}$, we consider the (upper) sum:
\begin{align}
\label{f}
{\Big(\overline{\sum_{t\in T}}f_t\Big)(x):=}
\limsup_{S\uparrow T,\;|S|<\infty}\; \sum_{t\in S}f_t(x), \quad x\in X.
\end{align}
\if{
Thus,
\begin{align*}
\limsup_{S\uparrow T,\;|S|<\infty} \sum_{t\in S}\al_t =\sup_{S\in\mathcal{F}(T)} \sup_{S'\in\mathcal{F}(T),\;S\subset S'} \sum_{t\in S'}\al_t.
\end{align*}
}\fi
If $\lim_{S\uparrow T,\;|S|<\infty}\; \sum_{t\in S}f_t(x)$ exists
{(in $[-\infty,+\infty]$)}
for all $x\in X$, we write simply {$\sum_{t\in T}f_t$}.

There are other ways of defining sums of infinite collections of functions; see, e.g.,
\emph{robust, limit} and \emph{Lebesgue} sums in \cite{HanJouVic23}.
The approach discussed in this paper is applicable with minor modifications to other definitions
{and, hopefully, also to the suprema of infinite collections; cf. \cite{CorHan24}}.

\if{
\PM{
``Reference [12] is seemingly not even available on the internet
(or hidden carefully). If the paper is not provided as an accessible preprint,
the reference must be removed.''}
}\fi

To simplify the presentation, we assume in the rest of the paper that $\dom\overline\sum_{t\in T}f_t{\ne\es}$ (hence, $\bigcap_{t\in T}\dom f_t\ne\es$), and $\inf\sum_{t\in S}f_t>-\infty$ for all $S\in\mathcal{F}(T)$.

The extensions of definitions \eqref{La0-} and \eqref{Th} to the infinite setting are {defined}
as follows:
\begin{align}
\label{La0}
{\Lambda}(\{f_{t}\}_{t\in T}):=&
\liminf_{S\uparrow T,\;|S|<\infty}\;
\liminf\limits_{\substack{\diam\{x_t\}_{t\in S}\to0}}\; \sum_{t\in S}f_{t}(x_{t}),
\\
\notag
{\Theta}(\{f_t\}_{t\in T})
:=&\limsup_{S\uparrow T,\;|S|<\infty}\; \limsup_{\substack{\diam\{x_t\}_{t\in S}\to0\\ x_t\in\dom f_t\,(t\in S)}}
\\
\label{Th0}
&\hspace{20mm}
\inf_{x\in X} \max\Big\{\max_{t\in S}d(x,x_t),\overline{\sum_{t\in T}}f_t(x)-\sum_{t\in S}f_{t}(x_{t})\Big\}.
\end{align}
They involve the infinite sum \eqref{f}, and the $\liminf$ and $\limsup$ operations \eqref{al} over the directed set \eqref{F(T)}.
We will sometimes write
${\Lambda}_X(\{f_{t}\}_{t\in T})$ and ${\Theta}_X(\{f_t\}_{t\in T})$ to specify the underlying space.
Definitions \eqref{La0} and \eqref{Th0}
obviously reduce to, respectively, \eqref{La0-} and \eqref{Th} when $T$ is finite.
We keep for \eqref{La0} the name \emph{uniform infimum} of $\{f_{t}\}_{t\in T}$.

The definitions below extend the corresponding ones discussed in the Introduction.
\begin{definition}
\label{D3.2}
The collection $\{f_{t}\}_{t\in T}$ is
\begin{enumerate}
\item
\label{D3.2.1}
\emph{uniformly lower semicontinuous} (on $X$) if
$\inf\overline\sum_{t\in T}f_{t} \le{\Lambda}(\{f_{t}\}_{t\in T})$;
\item
\label{D3.2.2}
\emph{firmly uniformly lower semicontinuous} (on $X$) if ${\Theta}(\{f_t\}_{t\in T})=0$.
\end{enumerate}
\end{definition}

By \eqref{f} and \eqref{La0}, we have
\begin{gather}
\label{9}
{\Lambda}(\{f_{t}\}_{t\in T}) \le\liminf_{S\uparrow T,\;|S|<\infty}\inf\sum_{t\in S}f_{t} \le\inf\overline{\sum_{t\in T}}f_{t},
\end{gather}
and consequently, the inequality in Definition~\ref{D3.2}\,\ref{D3.2.1} can only hold as equality.
If $T$ is finite, the second inequality in \eqref{9} holds as equality, while the first one
can  be strict.

\begin{example}
\label{E2.1}
Let $X=\R$, $f_1(x)=1/x$, $f_2(x)=-1/x$ for all $x\ne0$ and $f_1(0)=f_2(0){=+\infty}$.
Then $(f_1+f_2)(x)=0$ for all $x\ne0$ and $(f_1+f_2)(0)=+\infty$.
Hence, $\inf(f_1+f_2)=0$.
Set $x_{1k}:=1/k$ and $x_{2k}:=1/k^2$ for all $k\in\N$.
Then $|x_{1k}-x_{2k}|\to0$ and $f_1(x_{1k})+f_2(x_{2k})=k-k^2\to-\infty$ as $k\to+\infty$.
Hence, $\Lambda(f_1,f_2)=-\infty$.
\end{example}

In view of definitions \eqref{La0} and \eqref{Th0}, we have
\begin{gather*}
{\Theta}(\{f_t\}_{t\in T}) \ge\inf\overline\sum_{t\in T}f_{t} -{\Lambda}(\{f_{t}\}_{t\in T});
\end{gather*}
hence the next proposition holds true.

\begin{proposition}
\label{P3.2}
If $\{f_{t}\}_{t\in T}$ is firmly uniformly lower semicontinuous, then it is uniformly lower semicontinuous.
\end{proposition}

In general, the implication in Proposition~\ref{P3.2} is strict; see Example~\ref{nonuniform}.

\begin{remark}
\label{R3.2}
\begin{enumerate}
\item
Employing \eqref{La0-}, definition \eqref{La0} can be rewritten in a shorter form:
\begin{gather}
\label{R3.2-2}
{\Lambda}(\{f_{t}\}_{t\in T})=
\liminf_{S\uparrow T,\;|S|<\infty}\;
{\Lambda}(\{f_t\}_{t\in S}).
\end{gather}
Definition \eqref{Th0} cannot be in general rewritten as the (upper) limit of ${\Theta}(\{f_t\}_{t\in S})$ because the two sums in its \RHS\ are over different index sets.
\item
\label{R3.2.3}
Denote
\begin{gather}
\label{De0}
{\Delta}(\{f_{t}\}_{t\in T}):=
\limsup_{S\uparrow T,\;|S|<\infty}\Big(\inf\sum_{t\in S} f_t- {\Lambda}\big(\{f_t\}_{t\in S}\big)\Big).
\end{gather}
Thanks to our standing conventions, the \RHS\ of \eqref{De0} is well defined.
In view of \eqref{f} and \eqref{R3.2-2}, we have
\begin{gather}
\label{10}
{\Delta}(\{f_{t}\}_{t\in T})\le\inf\overline\sum_{t\in T} f_t- {\Lambda}(\{f_t\}_{t\in T}),
\end{gather}
and consequently, the inequality in Definition~\ref{D3.2}\,\ref{D3.2.1} implies
\begin{gather}
\label{R3.2-4}
{\Delta}(\{f_{t}\}_{t\in T})\le0.
\end{gather}
The latter condition defines a kind of \emph{weak} uniform lower semicontinuity.
If $T$ is finite, inequality \eqref{10} holds as equality, and the two properties coincide.
Observe from \eqref{De0} that ${\Delta}(\{f_{t}\}_{t\in T})\ge0$; hence, the mentioned weak property corresponds to the equality ${\Delta}(\{f_{t}\}_{t\in T})=0$.
\item
\label{R3.2.4}
In view of \eqref{Th} and \eqref{Th0}, we have
\begin{gather}
\label{R3.2-5}
{\Theta}(\{f_{t}\}_{t\in T})\ge
\limsup_{S\uparrow T,\;|S|<\infty}\;
{\Theta}(\{f_t\}_{t\in S}),
\end{gather}
and consequently, the equality in Definition~\ref{D3.2}\,\ref{D3.2.2} implies \begin{gather}
\label{R3.2-6}
\limsup\limits_{S\uparrow T,\;|S|<\infty}\,
{\Theta}(\{f_t\}_{t\in S})=0.
\end{gather}
The latter condition defines a kind of \emph{weak} firm uniform lower semicontinuity.
If $T$ is finite, inequality \eqref{R3.2-5} holds as equality, and the two properties coincide.
\end{enumerate}
\end{remark}

The following estimates are immediate consequences of \eqref{R3.2-2} and \eqref{De0}:
\begin{align}
\label{14}
\liminf_{S\uparrow T,\;|S|<\infty}\inf\sum_{t\in S} f_t\le
{\Lambda}(\{f_t\}_{t\in T})+
{\Delta}(\{f_{t}\}_{t\in T})\le \limsup_{S\uparrow T,\;|S|<\infty}\inf\sum_{t\in S} f_t.
\end{align}
They yield a characterization of the uniform lower semicontinuity.

\begin{proposition}
\label{P3.4}
The collection $\{f_{t}\}_{t\in T}$ is uniformly lower semicontinuous if and only if ${\Delta}(\{f_{t}\}_{t\in T})\le0$ and
$\overline\sum_{t\in T}f_{t}$ is $\inf$-stable
in the sense that
\begin{gather}
\label{P3.4-1}
\inf\overline\sum_{t\in T}f_{t} \le\liminf_{S\uparrow T,\;|S|<\infty}\inf\sum_{t\in S}f_{t}.
\end{gather}
\end{proposition}

\begin{proof}
Let ${\Delta}(\{f_{t}\}_{t\in T})\le0$ and
condition \eqref{P3.4-1} be satisfied.
Using the first inequality in \eqref{14}, we obtain
\begin{gather*}
\inf\overline{\sum_{t\in T}} f_t- {\Lambda}(\{f_t\}_{t\in T})\le
\liminf_{S\uparrow T,\;|S|<\infty}\inf\sum_{t\in S}f_{t}- {\Lambda}(\{f_t\}_{t\in T})\le
{\Delta}(\{f_{t}\}_{t\in T})\le0;
\end{gather*}
hence, $\{f_{t}\}_{t\in T}$ is uniformly lower semicontinuous.
Conversely, let $\{f_{t}\}_{t\in T}$ be uniformly lower semicontinuous.
It follows from the second inequality in \eqref{14} that
\begin{align*}
{\Delta}(\{f_{t}\}_{t\in T})\le \limsup_{S\uparrow T,\;|S|<\infty}\inf\sum_{t\in S} f_t-
{\Lambda}(\{f_t\}_{t\in T}) \le\inf\overline{\sum_{t\in T}} f_t-
{\Lambda}(\{f_t\}_{t\in T})\le0.
\end{align*}
Furthermore, the inequalities in \eqref{9} must hold as equalities, which implies \eqref{P3.4-1}.
\qed\end{proof}

\begin{remark}
The
{$\inf$-stability of the infinite sum}
property which plays an important role in Proposition~\ref{P3.4} is a particular case of a more general property defined for an arbitrary family of functions $\varphi_{S}:X\to\R_{\infty}$ $(S\in\mathcal{F}(T))$ and the corresponding $\limsup$ function
\begin{align}
\label{tf}
\varphi_\infty(x):=\limsup_{S\uparrow T,\;|S|<\infty} \varphi_S(x), \quad x\in X,
\end{align}
by the inequality
\begin{align}
\label{R3.5-1}
\inf\varphi_\infty\le\liminf_{S\uparrow T,\;|S|<\infty}\inf\varphi_S.
\end{align}
\if{
\AK{25/07/23.
Is this definition meaningful in infinite settings?
}
\AK{6/08/23.
Could it make sense reformulating the definition for the case of Lebesgue sums?}
\AK{10/08/23.
Some ``calculus'' of
\red{$\inf$-stability}
is needed.}
}\fi
Inequality \eqref{R3.5-1} (hence, also \eqref{P3.4-1})
can only hold as equality since the opposite inequality is always true.
Indeed, for any $x\in X$, in view of \eqref{tf}, we have
\begin{align*}
\liminf_{S\uparrow T,\;|S|<\infty}\inf\varphi_S \le\limsup_{S\uparrow T,\;|S|<\infty}\inf\varphi_S\le \limsup_{S\uparrow T,\;|S|<\infty}\varphi_S(x)=\varphi_\infty(x).
\end{align*}
Taking infimum over $x\in X$, we arrive at $\liminf_{S\uparrow T,\;|S|<\infty}\inf \varphi_S\le\inf\varphi_\infty$.

If $T$ is finite, then \eqref{R3.5-1} (hence, also \eqref{P3.4-1}) is automatically satisfied because ${T\in\mathcal{F}(T)}$ and $\varphi_\infty=\varphi_T$.
In general, it can be violated;
see \cite[p.~239]{RocWet98}.

If condition \eqref{R3.5-1} is satisfied and the infimum in the \LHS\ is attained at some point $\bx$, then, by~\eqref{tf},
$\limsup_{S\uparrow T,\;|S|<\infty}\varphi_S(\bx) =\varphi_\infty(\bx)\le\liminf_{S\uparrow T,\;|S|<\infty}\varphi_S(\bx)$, and consequently,
$\varphi_\infty(\bx)=\lim_{S\uparrow T,\;|S|<\infty}\varphi_S(\bx)$.
\end{remark}


Next, we provide sufficient conditions for the uniform and firm uniform lower semicontinuity properties,
some of which
extend the corresponding results established in \cite{FabKruMeh24} for the case $T=\{1,2\}$.

The next proposition is an immediate consequence of Definition~\ref{D3.2}\,\ref{D3.2.2}.
\begin{proposition}
Let $U_t\subset X$ $(t\in T)$.
If $\bigcap_{t\in T}U_t\ne\es$, then the family of indicator functions $\{i_{U_t}\}_{t\in T}$ is uniformly lower semicontinuous.
\end{proposition}

We refer to \cite[Proposition 5.13\,(i)]{FabKruMeh24} for a characterization of firm uniform lower semicontinuity of a finite family of indicator functions.
The situation when only some of the functions in a collection are indicator functions can also be of interest as it models constrained optimization problems.
Some sufficient conditions for uniform lower semicontinuity in the finite setting can be found in \cite{KruMeh22,FabKruMeh24}.

The next proposition shows that adding a constant function to a given collection does not affect its uniform and firm uniform lower semicontinuity properties.

\begin{proposition}
\label{P2.6}
Let $t_0\notin T$, $f_{t_0}(x):=c$ for some $c\in\R$ and all $x\in X$.
If $\{f_t\}_{t\in T}$ is uniformly (firmly uniformly) lower semicontinuous, then $\{f_t\}_{t\in T\cup\{t_0\}}$ is uniformly (firmly uniformly) lower semicontinuous.
\end{proposition}

\begin{proof}
Observe that $S'\in\mathcal{F}(T\cup\{t_0\})$ if and only if $S'\setminus\{t_0\}\in\mathcal{F}(T)$.
For any $S\in\mathcal{F}(T)$, $x\in X$ and $\{x_t\}_{t\in S\cup\{t_0\}}\subset X$, we have $\sum_{t\in S\cup\{t_0\}}f_t(x)=\sum_{t\in S}f_t(x)+c$ and $\sum_{t\in S\cup\{t_0\}}f_t(x_t)=\sum_{t\in S}f_t(x_t)+c$.
As a consequence, $\overline{\sum}_{t\in T\cup\{t_0\}}f_t=\overline{\sum}_{t\in T}f_t+c$, ${\Lambda}(\{f_{t}\}_{t\in T\cup\{t_0\}})= {\Lambda}(\{f_{t}\}_{t\in T})+c$ and
${\Theta}(\{f_{t}\}_{t\in T\cup\{t_0\}})= {\Theta}(\{f_{t}\}_{t\in T})$,
and the assertions follow.
\qed\end{proof}

\begin{remark}
Similar facts are also true for the weaker versions of uniform lower semicontinuity and firm uniform lower semicontinuity mentioned in Remark~\ref{R3.2}\,\ref{R3.2.3} and \ref{R3.2.4}, respectively, as well as the
{$\inf$-stability}
property defined by \eqref{P3.4-1}.
\end{remark}


The firm uniform lower semicontinuity property in Definition~\ref{D3.2}\,\ref{D3.2.2} is stable under uniformly continuous perturbations of one of the functions.

\begin{proposition}
\label{P3.7}
Suppose that $\{f_{t}\}_{t\in T}$ is firmly uniformly lower semicontinuous and $g\colon X\to\R$ is uniformly continuous.
Let $t_0\in T$.
Set $\tilde f_{t_0}:=f_{t_0}+g$ and $\tilde f_{t}:=f_{t}$ for all $t\in T\setminus\{t_0\}$.
Then the collection $\{\tilde f_t\}$ is firmly uniformly lower semicontinuous.
\end{proposition}

\begin{proof}
Let $\eps>0$.
Then there exists a $\rho>0$ such that $|g(x')-g(x'')|<\eps/2$ for all $x',x''\in X$ with $d(x',x'')<\rho$.
Set $\eps':=\min\{\eps/2,\rho\}$.
By \eqref{Th0}, for any $S_0\in\mathcal{F}(T)$, there exist an $S\in\mathcal{F}(T)$ with $S_0\cup\{t_0\}\subset S$ and a $\de>0$ such that, for any $x_{t}\in\dom f_t$
$(t\in S)$, with $\diam\{x_{t}\}_{t\in S}<\de$, one can find an $x\in X$ such that
\begin{align}
\label{P3.7P1}
\max_{t\in S}d(x,x_{t})<\eps'<\eps
\AND
\overline{\sum_{t\in T}}f_t(x)-\sum_{t\in S}f_{t}(x_{t})< \eps'\le\eps/2.
\end{align}
Then $d(x,x_{t_0})<\eps'\le\rho$.
Observe that $\dom\tilde f_t=\dom f_t$ for all $t\in S$, and
\begin{align}
\label{P3.7P2}
\overline{\sum_{t\in T}}\tilde f_t(x)-\sum_{t\in S}\tilde f_{t}(x_{t})=\overline{\sum_{t\in T}}f_t(x)-\sum_{t\in S}f_{t}(x_{t})+g(x)-g(x_{t_0})<\eps.
\end{align}
Thus, $\{\tilde f_t\}$ is firmly uniformly
lower semicontinuous.
\qed
\end{proof}

\if{
\AK{7/07/23.
The assertion can be easily extended to the case of a finite number of uniformly continuous perturbations.
One can even consider perturbations of all functions, but for that one needs to assume a kind of ``uniform (over $t$) uniform continuity''. :-)
I am not sure if it worth to introduce another definition.
}
}\fi

Thanks to Proposition~\ref{P2.6}, the next statement is an immediate consequence of Proposition~\ref{P3.7}.

\begin{corollary}
\label{C2.1}
Suppose that $\{f_t\}_{t\in T}$ is firmly uniformly
lower semicontinuous, $t_0\notin T$, $f_{t_0}\colon X\to\R$ is uniformly continuous.
Then the collection $\{f_t\}_{t\in T\cup\{t_0\}}$ is firmly uniformly
lower semicontinuous.
\end{corollary}

\begin{remark}
\begin{enumerate}
\item
Facts similar to those in Proposition~\ref{P3.7} and Corollary~\ref{C2.1} are also true for the weaker version of firm uniform lower semicontinuity mentioned in Remark~\ref{R3.2}\,\ref{R3.2.3}.
\item
The assumption that the uniform lower semicontinuity is firm plays an important role in the proof of Proposition~\ref{P3.7}.
The analogues of Proposition~\ref{P3.7} and Corollary~\ref{C2.1} do not necessarily hold for the
{non-firm}
uniform lower semicontinuity property (see Example~\ref{nonuniform} below).
\end{enumerate}
\end{remark}

\begin{example}
\label{nonuniform}
Let $X:=\R^3$, $x:=(u_1,u_2,u_3)$,
$$
f_1(x):=
\begin{cases}
0& \text{if }u_1\ne0,\;\frac1{u_1}-u_2\le0,\\
+\infty& \text{otherwise},
\end{cases}
\;\;
f_2(x):=
\begin{cases}
0& \text{if }u_1\ne0,\;\frac1{u_1}+u_3\ge0,\\
+\infty& \text{otherwise}.
\end{cases}
$$
Then $\inf(f_1+f_2)=\Lambda(f_1,f_2)=0$, i.e., the pair $\{f_1,f_2\}$ is uniformly lower semicontinuous.
Suppose it is firmly uniformly lower semicontinuous, i.e., $\Theta(f_1,f_2)=0$.
Take $x_{1k}:=(1/k,k,-k^2)\in\dom f_1$ and $x_{2k}:=(1/k^2,k,-k^2)\in\dom f_2$ $(k\in\N)$.
We obviously have $f_1(x_{1k})=f_2(x_{2k})=0$ $(k\in\N)$ and $\|x_{1k}-x_{2k}\|\to0$ as $k\to+\infty$.
Let $x_k:=(\al_k,\be_k,\ga_k)\in\dom(f_1+f_2)$ $(k\in\N)$.
Then $(f_1+f_2)(x_{k})=0$ $(k\in\N)$ and, by definition \eqref{Th0}, $\|x_{k}-x_{1k}\|\to0$ as $k\to+\infty$.
Thus, $\al_k\ne0$, $1/\al_k-\be_k\le0$, $1/\al_k+\ga_k\ge0$ $(k\in\N)$, $\al_k-1/k\to0$,
$\be_k-k\to0$ and $\ga_k+k^2\to0$ as $k\to+\infty$, and consequently, $\be_k+\ga_k\ge0$ for all $k\in\N$, and $\be_k+\ga_k+k^2-k\to0$ as $k\to+\infty$, which is impossible.
Hence, $\Theta(f_1,f_2)>0$, i.e., $\{f_1,f_2\}$ is not firmly uniformly lower semicontinuous.

Adding a third (even linear) function to the pair $\{f_1,f_2\}$ destroys its uniform lower semicontinuity.
Set $f_3(x):=u_2+u_3$.
Then $f_3(x)\ge0$ for all $x\in\dom f_1\cap\dom f_2$, and $\inf(f_1+f_2+f_3)=f_3(1,1,-1)=0$.
Take $x_{3k}:=(0,k,-k^2)$ $(k\in\N)$.
Then $\diam\{x_{1k},x_{2k},x_{3k}\}{\to0}$ and $f_1(x_{1k})+f_2(x_{2k})+f_3(x_{3k})=k-k^2\to-\infty$ as $k\to+\infty$.
Hence, $\Lambda(f_1,f_2,f_3)=-\infty<\inf(f_1+f_2+f_3)$, i.e., the triple $\{f_1,f_2,f_3\}$ is not uniformly lower semicontinuous.
\end{example}
\if{
\olive{\begin{example}
Given two proper lower semicontinuous functions $f_{1},$ $f_{2}:X\rightarrow \mathbb{R}_{\infty }$, we suppose that
$\mathrm{dom}f_1\cap \mathrm{dom}f_2\ne \emptyset$, and the collection $\{f_1,\ f_2\}$ is not uniformly lower semicontinuous on $X$; that is,
$$
\inf_{X}f_{1}+f_{2} > \Lambda (\{f_{1},\ f_{2}\}).
$$
Next, we verify that
\begin{equation*}
\inf_{X}f_{1}+f_{2}=\inf_{X\times \mathbb{R}^{2}}f+g+h,\quad \Lambda(f_{1},f_{2})=\Lambda(f,g,h),
\end{equation*}
where $f,$ $g,$ $h:X\times \mathbb{R}^{2}\rightarrow \mathbb{R}_{\infty }$
are the lower semicontinuous functions defined as
\begin{equation*}
f(x,\lambda _{1},\lambda _{2}):=\lambda _{1}+\lambda _{2},\text{ }%
g(x,\lambda _{1},\lambda _{2}):=\mathrm{I}_{\mathrm{epi}f_{1}}(x,\lambda
_{1}),\text{ }h(x,\lambda _{1},\lambda _{2}):=\mathrm{I}_{\mathrm{epi}%
f_{2}}(x,\lambda _{2}).
\end{equation*}
Observe that the function $f$ is linear (and continuous).
So,
$$
\inf_{X\times \mathbb{R}^{2}}f+g+h > \Lambda(f,g,h),
$$
and the collection $\{f,g,h\}$ is not uniformly lower semicontinuous on $X$. However, the subcollection $\{g,h\}$ is uniformly lower semicontinuous on $X$, because $g$, $h$ are indicator functions and satisfy $\mathrm{epi}f_{1}\cap \mathrm{epi}f_{2}\ne \emptyset$.
\end{example}
}
}\fi


The next proposition
exploits inf-compactness assumptions in order to guarantee
conditions \eqref{R3.2-4} and \eqref{R3.2-6}, which are essential for the uniform and firm uniform lower semicontinuity properties, respectively; see
Remark~\ref{R3.2}\,\ref{R3.2.3} and \ref{R3.2.4}, and Proposition~\ref{P3.4}.


\begin{proposition}
\label{P4.6}
Suppose that the functions $f_t$ $(t\in T)$ are
\lsc.
If there exist a $t_0\in T$ and an $S_0\in\mathcal{F}(T)$ such that the sets $\{x\in X\mid f_{t_0}(x)\leq c\}$ are
compact for all $c\in\R$, and $\Lambda(\{f_t\}_{t\in S\setminus\{t_0\}})>-\infty$
for all $S\in\mathcal{F}(T)$ with $S_0\cup\{t_0\}\subset S$,
then condition \eqref{R3.2-4} is satisfied.
If, additionally,
\begin{align}
\label{P4.6-3}
\limsup_{\substack{\diam\{x_t\}_{t\in S}\to0\\ x_t\in\dom f_t\;(t\in S)}}\; \sum_{t\in S}f_{t}(x_t)<+\infty
\end{align}
for each $S\in\mathcal{F}(T)$ with $S_0\cup\{t_0\}\subset S$,
then
condition \eqref{R3.2-6} is satisfied.
\if{
{\AH{11/06/2025. Is not "firmly uniformly lower semicontinuous" (not weakly)?}}
\AK{12/06/25.
We show that $\limsup_{S\uparrow T,\;|S|<\infty}\,
{\Theta}(\{f_t\}_{t\in S})=0$.
This is weaker than ${\Theta}(\{f_t\}_{t\in T})=0$; see \eqref{R3.2-5}.}
}\fi
\end{proposition}

\begin{proof}
Let $t_0$ and $S_0$ be as in the proposition.
Take any $S\in\mathcal{F}(T)$ such that $S_0\cup\{t_0\}\subset S$, and any $x_{tk}\in\dom f_t$ ($t\in S$, $k\in\N$) such that $\diam\{x_{tk}\}_{t\in S}{\to0}$ as $k\to+\infty$.

\underline{Case 1}.
Let $\al:=\limsup_{k\to\infty}\sum_{t\in S}f_{t}(x_{tk}) <+\infty$
(this is automatically true if condition \eqref{P4.6-3} is satisfied).
Then
\begin{align*}
\limsup_{k\to\infty}f_{t_0}(x_{t_0k}) +\Lambda(\{f_t\}_{t\in S\setminus\{t_0\}})\le\al,
\end{align*}
and consequently,
$f_{t_0}(x_{t_0k})<c:= \al-{\Lambda(\{f_t\}_{t\in S\setminus\{t_0\}})}
+1<+\infty$ for all sufficiently large $k\in\N$,
i.e., a tail of the sequence $\{x_{t_0k}\}$
belongs to $\{x\in X\mid f_{t_0}(x)\leq c\}$.
By the compactness assumption, we can assume that $x_{t_0k}\to\hat x\in X$, and consequently,
$x_{tk}\to\hat x$ {for all $t\in S$}.
Thanks to the lower semicontinuity of $f_t$ $(t\in S)$, we have
\begin{align}
\label{P4.6P1}
\lim_{k\to+\infty}\sum_{t\in S}(f_{t}(\hat x)-f_{t}(x_{tk}))\le0.
\end{align}
\if{
{\AH{25/06/2025. Perhaps we should also check that $\hat x\in\bigcap_{t\in S}\dom f_t $. Actually, this could be done as in line (319) above by repeating the same argument while using $x_{tk}$ instead of $x_{t_0k}$.
}}
\AK{12/06/25.
$\hat x\in\bigcap_{t\in S}\dom f_t $ thanks to the lower semicontinuity of the functions.}
}\fi
Observe that $\hat x\in\bigcap_{t\in S}\dom f_t$ since, by the assumption, $\limsup_{k\to\infty}\sum_{t\in S}f_{t}(x_{tk}) <+\infty$.

\underline{Case 2}.
If $\sum_{t\in S}f_{t}(x_{tk})\to+\infty$ as $k\to+\infty$
(thus, condition \eqref{P4.6-3} is not satisfied), then, for any $\hat x\in\bigcap_{t\in S}\dom f_t$, condition \eqref{P4.6P1} trivially holds.

Combining the two cases, we
arrive at \eqref{R3.2-4}.
Moreover, under condition \eqref{P4.6-3},
only the second case is possible and, as shown above, \begin{align*}
{\Theta}(\{f_t\}_{t\in S}) \le\lim_{k\to+\infty}\max\Big\{\max_{t\in S}d(\hat x_S,x_{tk}),\sum_{t\in S}(f_{t}(\hat x_S)-f_{t}(x_{tk}))\Big\}=0.
\end{align*}
Hence,
condition \eqref{R3.2-6} is satisfied.
\qed
\end{proof}
\if{
\begin{remark}
If $X$ is sequentially compact, the compactness assumptions in Proposition~\ref{P4.6} are satisfied automatically.
{\AH{11/06/2025. If you agree with the observation above, you can remove the word "sequentially" here.}}
\end{remark}
}\fi

The arguments in the proof of Proposition~\ref{P4.6} can be reused to prove the next statement.

\begin{proposition}
\label{P4.7}
Let $X$ be a normed space,
and the functions $f_t$ $(t\in T)$ be weakly sequentially \lsc\ on $U\subset X$.
If there exist a $t_0\in T$ and an $S_0\in\mathcal{F}(T)$ such that the sets $\{x\in U\mid f_{t_0}(x)\leq c\}$ are weakly sequentially compact for all $c\in\R$, and {$\Lambda_{U}(\{f_t\}_{t\in S\setminus\{t_0\}})>-\infty$}
for all $S\in\mathcal{F}(T)$ with $S_0\cup\{t_0\}{\subset S}$,
then
\begin{gather*}
\limsup_{S\uparrow T,\;|S|<\infty}\Big(\inf\sum_{t\in S} f_t- {\Lambda}_U\big(\{f_t\}_{t\in S}\big)\Big)\le 0.
\end{gather*}
%
\end{proposition}
\if{
{\AH{11/06/2025. Here the word "sequentially" makes sense.}}
{\AH{11/06/2025. Remember Eberlein–Šmulian theorem.}}
}\fi

\begin{remark}
\begin{enumerate}
\item
Thanks to Remark~\ref{R3.2}\,\ref{R3.2.3} and \ref{R3.2.4}, if $T$ is finite, the first part of Proposition~\ref{P4.6} and Proposition~\ref{P4.7} provide sufficient conditions for the uniform lower semicontinuity, while the second part of Proposition~\ref{P4.6} provides sufficient conditions for the firm uniform lower semicontinuity.
Thanks to Proposition~\ref{P4.6}, under the  inf-stability condition \eqref{P3.4-1}, the mentioned sufficient conditions guarantee the uniform lower semicontinuity even when $T$ is infinite.
\item
In Proposition~\ref{P4.7}, if $X$ is Banach, then, by Eberlein–Šmulian theorem, it suffices to assume the level sets to be weakly compact.
\end{enumerate}
\end{remark}
\if{
{\AH{11/06/2025. Have we defined the weak firm uniform lower semicontinuity property?}}
\AK{12/06/25.
Both ``weak'' properties are defined in Remark~\ref{R3.2}\,\ref{R3.2.3} and \ref{R3.2.4} without making a strong emphasis.
As we discussed earlier, I am hiding all less essential things into remarks.
The ``weak'' properties ARE essential in the FINITE case.
I am thinking about ways of discussing these things without formally defining the ``weak'' properties.}
}\fi

\section{Optimality conditions}
\label{S6}

In this section, we illustrate the extended decoupling techniques discussed in Section~\ref{S3} by employing them to deriving dual
necessary conditions for a point $\bar x\in\dom\overline\sum_{t\in T}f_{t}$ to be a
local minimum of the sum $\overline\sum_{t\in T}f_{t}$ of \lsc\ functions $f_{t}$ $(t\in T)$.
Since we target dual necessary conditions, the underlying space in this section is assumed Banach.

\begin{theorem}
\label{T6.1}
Let $X$ be a Banach space, and
$\bx\in X$.
Suppose that there is a number ${\de>0}$ such that
$\bx$ is a minimum of $\overline\sum_{t\in T}f_{t}$ over $B_\de(\bx)$;
$f_t$ $(t\in T)$ are lower semicontinuous and bounded from below on $B_\de(\bx)$; and $\{f_{t}\}_{t\in T}$ is uniformly lower semicontinuous on $B_\de(\bx)$.
\if{
\begin{itemize}
\item
$\bx$ is a minimum of $\overline\sum_{t\in T}f_{t}$ over $B_\de(\bx)$;
\item
$f_t$ $(t\in T)$ are lower semicontinuous and bounded from below on $B_\de(\bx)$;
\item
$\{f_{t}\}_{t\in T}$ is quasiuniformly lower semicontinuous on $B_\de(\bx)$.
\end{itemize}
}\fi
\begin{enumerate}
\item
\label{T6.1.1}
For any $\eps>0$ and $S_0\in\mathcal{F}(T)$,
there exist an $S:=\{t_1,\ldots,t_m\}\in\mathcal{F}(T)$ with some $m\in\N$, and points $x_1,\ldots,x_m\in B_\eps(\bx)$
such that $S_0\subset S$, and
\begin{gather}
\label{T6.2-02}
\sum_{i=1}^m\big(f_{t_i}(x_i)-f_{t_i}(\bx)\big)\le0,
\\
\label{T6.2-03}
0\in\sum_{i=1}^m{\sdc}f_{t_i}(x_i)+\eps\B^*.
\end{gather}
\item
\label{T6.1.2}
If $X$ is Asplund, then,
for any $\eps>0$ and $S_0\in\mathcal{F}(T)$,
there exist an $S:=\{t_1,\ldots,t_m\}\in\mathcal{F}(T)$ with some $m\in\N$, and points $x_1,\ldots,x_m\in B_\eps(\bx)$
such that $S_0\subset S$, and
\begin{gather}
\label{T6.2-01}
\sum_{i=1}^m\big(f_{t_i}(x_i)-f_{t_i}(\bx)\big)<\eps,
\\
\label{T6.2-05}
0\in\sum_{i=1}^m{\sdf}f_{t_i}(x_i)+\eps\B^*.
\end{gather}
\end{enumerate}
\end{theorem}

\begin{proof}
Fix any $\eps>0$
and $S_0\in\mathcal{F}(T)$.
By Definition~\ref{D3.2}\,\ref{D3.2.1} and definition \eqref{La0} with $B_\de(\bar x)$ in place of $X$,
we have
\begin{gather}
\label{T6.1P00}
\overline{\sum_{t\in T}}f_{t}(\bx)\le\liminf_{S\uparrow T,\;|S|<\infty}\; \liminf_{\substack{\diam\{x_t\}_{t\in S}\to0\\ x_t\in B_\de(\bar x)\;(t\in S)}}\;\sum_{t\in S}f_{t}(x_{t}).
\end{gather}
Set $\de':=\min\{\eps,\de\}$.
Choose a number $\eps'\in(0,\eps\de'/2)$.
Observe that $2\eps'/\eps<\de'$.
Choose numbers $\rho\in(2\eps'/\eps,\de')$
and $\eta>0$.
{Observe that}
$2\eps'/\rho<\eps$.
Set
\begin{align}
\label{T6.1P01}
\al:=\eps'/\rho^2\AND
\xi:=\eps-2\eps'/\rho.
\end{align}
By \eqref{f} and \eqref{T6.1P00}, there exist an
$S:=\{t_1,\ldots,t_m\}\in\mathcal{F}(T)$ with some $m\in\N$, and a number $\eta'\in(0,\eta)$ such that
$S_0\subset S$, and
\begin{gather}
\label{T6.1P003a}
\Big|\sum_{i=1}^mf_{t_i}(\bx)-\overline{\sum_{t\in T}}f_{t}(\bx)\Big|<\eps'/2,
\\
\label{T6.1P003b}
\overline{\sum_{t\in T}}f_{t}(\bx)< \inf_{\substack{\diam\{u_1,\ldots,u_m\}<\eta'\\ u_1,\ldots,u_m\in\overline B_\rho(\bar x)}}\; \sum_{i=1}^m f_{t_i}(u_{i})+\eps'/2.
\end{gather}
As a consequence of \eqref{T6.1P003a} and \eqref{T6.1P003b}, we have
\begin{align}
\label{T6.1P03}
\sum_{i=1}^mf_{t_i}(\bx) <\inf_{\substack{\diam\{u_1,\ldots,u_m\}<\eta'\\ u_1,\ldots,u_m\in\overline B_\rho(\bar x)}}\; \sum_{i=1}^m f_{t_i}(u_{i})+\eps'.
\end{align}
Let $c:=\inf_{i=1,\ldots,m;\,u\in B_\de(\bx)}f_{t_i}(u)$ ($>-\infty$).
Choose a
number $\ga>\big(\sum_{i=1}^mf_{t_i}(\bx)-mc\big)/\eta'$.
Consider the functions
$\varphi_{\gamma},\widehat\varphi_{\gamma}\colon X^m\to\R_{\infty}$
defined
for all $u_1,\ldots,u_m\in X$ by
\begin{align}
\label{phiga}
&\varphi_\ga(u_1,\ldots,u_m):= \sum_{i=1}^mf_{t_i}(u_i)+\ga\,\diam\{u_1,\ldots,u_m\},
\\
\label{T6.1P04}
&\widehat\varphi_{\gamma}(u_1,\ldots,u_m):=
\varphi_\ga(u_1,\ldots,u_m)+\al \max_{i=1,\ldots,m}\|u_i-\bx\|^2.
\end{align}
Observe that $\widehat\varphi_{\gamma}(\bx,\ldots,\bx) =\varphi_{\gamma}(\bx,\ldots,\bx) =\sum_{i=1}^mf_{t_i}(\bx)$,
and $\widehat\varphi_{\gamma}$ is \lsc\ and bounded from below on
$(\overline{B}_{\rho}(\bar x))^m$.
Noting that $(\overline{B}_{\rho}(\bar x))^m$ is a complete metric space (equipped with the metric induced by the max norm on $X^m$),
by Ekeland variational principle
(Lemma~\ref{Ekeland}),
applied to the restriction of $\widehat\varphi_{\ga}$ to $(\overline{B}_{\rho}(\bar x))^m$,
there exist
$\hat x_1,\ldots,\hat x_m\in
\overline{B}_{\rho}(\bar x)$
such that
\begin{align}
\label{T6.1P05-1}
&\widehat\varphi_{\ga}(\hat x_1,\ldots,\hat x_m)
\le \sum_{i=1}^mf_{t_i}(\bx),\\
\notag
\widehat\varphi_{\ga}(u_1,\ldots,u_m)+\xi \max_{i=1,\ldots,m}&\|u_i-\hat x_i\|
{>}
\widehat\varphi_{\ga}(\hat x_1,\ldots,\hat x_m)
\\
\label{T6.1P05-2}
&\text{for all}\;\; (u_1,\ldots,u_m)
\in(\overline{B}_{\rho}(\bar x))^m\setminus (\hat x_1,\ldots,\hat x_m).
\end{align}
In view of \eqref{phiga} and \eqref{T6.1P04},
condition \eqref{T6.1P05-1} yields
\begin{align}
\label{T6.1P06}
\sum_{i=1}^mf_{t_i}(\hat x_i)\le \sum_{i=1}^mf_{t_i}(\bx).
\end{align}
Moreover, it follows from \eqref{phiga}, \eqref{T6.1P04} and \eqref{T6.1P05-1}
that
\begin{align}
\notag
\ga\,\diam\{\hat x_1,\ldots,\hat x_m\} +\al\max_{i=1,\ldots,m} \|\hat x_i-\bx\|^2 =&\widehat\varphi_\ga(\hat x_1,\ldots,\hat x_m) -\sum_{i=1}^m f_{t_i}(\hat x_i)
\\
\label{T6.1P07}
\le&
\sum_{i=1}^mf_{t_i}(\bx)-\sum_{i=1}^m f_{t_i}(\hat x_i).
\end{align}
Hence, $\diam\{\hat x_1,\ldots,\hat x_m\}
<(\sum_{i=1}^mf_{t_i}(\bx)-mc)/\ga<\eta'$
and, thanks to \eqref{T6.1P03},
$\max_{i=1,\ldots,m}\|\hat x_i-\bx\|^2<\eps'/\al=\rho^2.$
Thus,
\begin{align}
\label{T6.1P09}
\max_{i=1,\ldots,m}\|\hat x_i-\bx\|<\rho.
\end{align}
In view of \eqref{T6.1P04}, we have
for all $u_1,\ldots,u_m\in X$:
\begin{align*}
\varphi_\ga(\hat x_1,\ldots,\hat x_m) &-\varphi_\ga(u_1,\ldots,u_m)
-\widehat\varphi_{\ga}(\hat x_1,\ldots,\hat x_m)+\widehat\varphi_{\ga}(u_1,\ldots,u_m)\\
&=
\al\max_{i=1,\ldots,m}\bigl(\|u_{i}-\bx\|^2-\|\hat x_i-\bx\|^2\bigr)
\\
&\le\al\max_{i=1,\ldots,m}\big(\|u_{i}-\hat x_i\|\big(\|u_{i}-\bx\|+\|\hat x_i-\bx\|\big)\big)\\
&\le
\al \max_{i=1,\ldots,m}\|u_{i}-\hat x_i\|
\max_{i=1,\ldots,m}\big(\|u_{i}-\bx\|+\|\hat x_i-\bx\|\big),
\end{align*}
and consequently, thanks to \eqref{T6.1P01} and \eqref{T6.1P05-2},
\begin{align}
\notag
\sup_{\substack{u_1,\ldots,u_m\in\overline{B}_{\rho}(\bar x)\\
(u_1,\ldots,u_m)\ne(\hat x_1,\ldots,\hat x_m)}}
\frac{\varphi_\ga(\hat x_1,\ldots,\hat x_m)-\varphi_\ga(u_1,\ldots,u_m)}
{\max_{i=1,\ldots,m}\|u_{i}-\hat x_i\|}
&<\xi+2\al\rho
\\&=
\label{T6.1P11}
\eps-\frac{2\eps'}{\rho}+\frac{2\eps'}{\rho}=\eps.
\end{align}
By \eqref{T6.1P09}, the points $\hat x_1,\ldots,\hat x_m$ lie in the interior of $\overline B_{\rho}(\bx)$,
and it follows from \eqref{T6.1P11} that
\begin{align}
\label{T6.1P12}
\limsup_{\substack{(u_1,\ldots,u_m)\to(\hat x_1,\ldots,\hat x_m)\\
(u_1,\ldots,u_m)\ne(\hat x_1,\ldots,\hat x_m)}}
\frac{\varphi_\ga(\hat x_1,\ldots,\hat x_m)-\varphi_\ga(u_1,\ldots,u_m)}
{\max_{i=1,\ldots,m}\|u_{i}-\hat x_i\|}<\eps,
\end{align}
and consequently, there is a number
$\hat\eps\in(0,\eps)$ such that
\begin{align}\label{Abderrahim}
\liminf\limits_{\substack{(u_1,\ldots,u_m)\to(\hat x_1,\ldots,\hat x_m)\\
(u_1,\ldots,u_m)\ne(\hat x_1,\ldots,\hat x_m)}}
\frac{\varphi_\gamma(u_1,\ldots,u_m)-\varphi_{\gamma}(\hat x_1,\ldots,\hat x_m)+ \hat\eps\max\limits_{i=1,\ldots,m}\|u_t-\hat x_i\|}
{\max\limits_{i=1,\ldots,m}\|u_{i}-\hat x_i\|}>0.
\end{align}
For all $u_1,\ldots,u_m\in
X$, set
\begin{gather}
\label{phi}
\varphi(u_1,\ldots,u_m):=\sum_{i=1}^mf_{t_i}(u_{i}),
\\
\notag
g(u_1,\ldots,u_m):=\ga\,\diam\{u_1,\ldots,u_m\},
\quad
h(u_1,\ldots,u_m):=\hat\eps
{\max_{i=1,\ldots,m}\|u_{i}-\hat x_i\|}.
\end{gather}
By definition of the Fr\'{e}chet subdifferential and \eqref{phiga}, inequality \eqref{Abderrahim} yields
\begin{align}
\label{T6.2P03}
0\in{\sdf}\left(\varphi+g+h\right)(\hat x_1,\ldots,\hat x_m).
\end{align}
The next step is to apply to \eqref{T6.2P03} a subdifferential sum rule.
Note that $g$ and $h$ are convex and Lipschitz continuous, and, for all $u_1,\ldots,u_m\in X$,
the following relations hold for the respective subdifferentials:
\begin{gather}
\label{T6.2P05}
\sdf\varphi(u_1,\ldots,u_m)\subset\prod_{i=1}^m\sdf f_{t_i}(u_i),\quad
\sdc\varphi(u_1,\ldots,u_m)\subset\prod_{i=1}^m\sdc f_{t_i}(u_i),\\
\label{T6.2P07}
(u_1^*,\ldots,u_m^*)\in\sd g(u_1,\ldots,u_m)
\quad\Rightarrow\quad
\sum_{i=1}^mu_{i}^*=0,
\\
\label{T6.2P08}
(u_1^*,\ldots,u_m^*)\in\sd h(u_1,\ldots,u_m)
\quad\Rightarrow\quad
\sum_{i=1}^{m}\|u_i^*\|\le\hat\eps.
\end{gather}
Implication \eqref{T6.2P07} is less common in the literature on the subject.
So, we give here a short proof.
Let $u_1,\ldots,u_m\in X$ and $(u_1^*,\ldots,u_m^*)\in\sd g(u_1,\ldots,u_m)$, i.e.,
\begin{gather}
\label{T6.2P09}
\ga\diam\{u_1',\ldots,u_m'\}-\ga\diam\{u_1,\ldots,u_m\} \ge\sum_{i=1}^m\ang{u_t^*,u_t'-u_t}
\end{gather}
for all $u_1',\ldots,u_m'\in X$.
Let $\zeta>0$.
There is a $u\in X$ such that
$\|u\|=1$ and $\ang{\sum_{i=1}^mu_{i}^*,u} >\|\sum_{i=1}^mu_i^*\|-\zeta$.
For all $i=1,\ldots,m$, set $u_{i}':=u_{i}+u$.
Then $\diam\{u_1',\ldots,u_m'\}=\diam\{u_1,\ldots,u_m\}$, $\sum_{i=1}^m\ang{u_{i}^*,u_{i}'-u_{i}}=\ang{\sum_{i=1}^mu_{i}^*,u}$, and it follows from \eqref{T6.2P09} that $\|\sum_{i=1}^mu_i^*\|<\zeta$.
Since $\zeta$ is arbitrary, we obtain $\sum_{i=1}^mu_i^*=0$.

\if{
\AK{27/07/23.
I have failed to find a reference for \eqref{T6.2P07}.}
\AH{A possible reference for \eqref{T6.2P07} would be Lemma 3.2.2 in Borwein-Zhu' book:\\
Jonathan M. Borwein, Qiji J. Zhu: Techniques of Variational Analysis, CMS Books in Mathematics, Springer, 2005.}
\AK{12/08/24.
Lemma 3.2.2 in Borwein-Zhu gives an analogue of \eqref{T6.2P07} but not the exact assertion.
Anyway, they write in the proof that it ``is easy and left as an exercise''.
I think this is applicable also to \eqref{T6.2P07}.}
\AH{I am not sure but the relations in \eqref{T6.2P07} (or one of them) only hold as inclusions "$\subset$", the one which serve us.}
\AK{12/08/24.
Did you mean \eqref{T6.2P05}?
I have replaced the equalities by inclusions.
They are straightforward.
The equalities may also be true but this is not important.}
}\fi
Inclusion \eqref{T6.2P03} obviously yields $0\in{\sdc}\left(\varphi+g+h\right)(\hat x_1,\ldots,\hat x_m)$.
By the Clarke subdifferential sum rule
(Lemma~\ref{SR}\,\ref{SR.Clarke}),
there exist subgradients $(x_1^*,\ldots,x_m^*)\in{\sdc}\varphi(\hat x_1,\ldots,\hat x_m)$, $(x_1'^*,\ldots,x_m'^*)\in{\sdf}g(\hat x_1,\ldots,\hat x_m)$ and $(x_1''^*,\ldots,x_m''^*)\in{\sdf}h(\hat x_1,\ldots,\hat x_m)$ such that $(x_1^*,\ldots,x_m^*)+(x_1'^*,\ldots,x_m'^*) +(x_1''^*,\ldots,x_m''^*)=0$.
By \eqref{T6.2P07} and \eqref{T6.2P08},
$$\Big\|\sum_{i=1}^mx_{i}^*\Big\|
=\Big\|\sum_{i=1}^m(x_{i}^*+x_{i}'{}^*)\Big\|
=\Big\|\sum_{i=1}^mx_{i}''{}^*\Big\|
\le\sum_{i=1}^m \|x_{i}''{}^*\|\le\hat\eps <\eps.$$
Set $x_i:=\hat x_i$ $(i=1,\ldots,m)$.
In view of \eqref{T6.1P06}, \eqref{T6.1P09} and \eqref{T6.2P05}, we have
$x_1,\ldots,x_m\in B_\eps(\bx)$, and conditions \eqref{T6.2-02} and \eqref{T6.2-03} are satisfied.

Suppose now that $X$ is an Asplund space.
By the \Fr\  subdifferential sum rule
(Lem\-ma~\ref{SR}\,\ref{SR.2})
combined with the convex sum rule
(Lemma~\ref{SR}\,\ref{SR.1}),
applied to \eqref{T6.2P03},
\if{
there exist a point $\{x_t\}_{t\in S}$
arbitrarily close to $\{\hat x_t\}_{t\in S}$
with $\varphi(\{x_t\}_{t\in S})$ arbitrarily close to $\varphi(\hat x_1,\ldots,\hat x_m)$
and a subgradient $\{x_t^*\}_{t\in S}\in{\sdf}\varphi(\{x_t\}_{t\in S})$
such that, taking into account \eqref{T6.1P09}, \eqref{T6.2P07} and \eqref{T6.2P08}, we have
\red{$\|x_t-\bx\|<\eps$ $(t\in S)$},
condition \eqref{T6.2-01} is satisfied, and $\|\sum_{i=1}^mx_{t}^*\|<\eps$.
In view of \eqref{phi} and \eqref{T6.2P05}, condition \eqref{T6.2-05} is satisfied too.
}\fi
for any $\zeta>0$,
there exist points $(x_1,\ldots,x_m),(x_1',\ldots,x_m')\in{X^m}$, and subgradients
$(x_1^*,\ldots,x_m^*)\in{\sdf}\varphi(x_1,\ldots,x_m)$, $(x_1'^*,\ldots,x_m'^*)\in{\sdf}g(x_1',\ldots,x_m')$ and $(x_1''^*,\ldots,x_m''^*)\in{\sdf}h(x_1',\ldots,x_m')$ such that
\begin{gather*}
\|x_i-\hat x_i\|<\zeta
\;\; (i=1,\ldots, m),\quad
|\varphi(x_1,\ldots,x_m)-\varphi(\hat x_1,\ldots,\hat x_m)|<\zeta,
\\
\sum_{i=1}^{m}\|x_i^*+x_i'^* +x_i''^*\|<\zeta.
\end{gather*}
The number $\zeta$ can be chosen small enough so that (thanks to \eqref{T6.1P09})
$\|x_i-\bx\|<\eps$ $(i=1,\ldots,m)$,
(thanks to \eqref{T6.1P06} and \eqref{phi}) condition \eqref{T6.2-01} is satisfied, and $\hat\eps+\zeta<\eps$.
Hence, thanks to \eqref{T6.2P05}, \eqref{T6.2P07} and \eqref{T6.2P08}, we see that
$$\Big\|\sum_{i=1}^mx_{t_i}^*\Big\|
=\Big\|\sum_{i=1}^m(x_{t_i}^*+x_{t_i}'{}^*)\Big\|
<\Big\|\sum_{i=1}^mx_{t_i}''{}^*\Big\|+\zeta
\le \sum_{i=1}^{m}\|x_{t_i}''^*\|+\zeta
\le\hat\eps+\zeta<\eps,$$
i.e., condition \eqref{T6.2-05} is satisfied.
\qed
\end{proof}

\begin{remark}
\label{R3.1}
\begin{enumerate}
\item
\label{R3.1.2}
Some sufficient conditions for the uniform lower semicontinuity
assumed in Theorem~\ref{T6.1} are given in
{Section~\ref{S3}}.
\item
The dual necessary conditions (multiplier rules) in Theorem~\ref{T6.1} hold not necessarily at the reference point but at some points arbitrarily close to it.
Such conditions are usually referred to as \emph{approximate} or \emph{fuzzy}.

\item
The two assertions in Theorem~\ref{T6.1} correspond to the two subdifferential sum rules in Lemma~\ref{SR}\,\ref{SR.Clarke} and \ref{SR.2} used in the proof.
In the first assertion, other ``trustworthy'' subdifferentials can be employed instead of the Clarke ones, e.g., $G$-subdifferentials of Ioffe; see \cite[Theorem~4.69]{Iof17}.

\item
\label{R3.1.4}
Conditions \eqref{T6.2-02} and \eqref{T6.2-01} in Theorem~\ref{T6.1} guarantee that the values $\sum_{i=1}^mf_{t_i}(x_i)$ are not (much) greater than $\sum_{i=1}^mf_{t_i}(\bx)$.
This is a little different from the traditional for this type of results (for a finite number of functions) conditions guaranteeing that the values of the individual functions $f_{t_i}$ at $x_i$ and $\bx$ are close.
Such conditions easily follow from \eqref{T6.2-02} and \eqref{T6.2-01} if one additionally assumes that the functions $f_t$ $(t\in T)$ are \emph{jointly lower semicontinuous} at $\bx$ in the following sense:
\begin{gather}
\label{below}
\inf_{\de>0}\;\lim_{S\uparrow T,\;|S|<\infty}\; \sum_{t\in S} \Big(f_t(\bx)-\inf_{B_\de(\bx)}f_t\Big)\le0.
\end{gather}
Note that the limit in \eqref{below} exists, since $\{\sum_{t\in S} (f_t(\bx)-\inf_{B_\de(\bx)}f_t)\}_{S\in\mathcal{F}(T)}$ is
nondecreasing.
If $T$ is finite, the joint lower semicontinuity is equivalent to the lower semicontinuity of all the individual functions $f_t$ $(t\in T)$, but in infinite settings, it is in general stronger.
It is not difficult to show that under the assumption of \eqref{below} conditions \eqref{T6.2-02} and \eqref{T6.2-01} can be replaced in Theorem~\ref{T6.1} by the more traditional ones: $|f_{t_i}(x_i)-f_{t_i}(\bx)|<\eps$ $(i=1,\ldots,m)$.

\item
It is not difficult to show that
the necessary conditions in Theorem~\ref{T6.1} characterize not only local
minimum points but also a broader class of
stationary points (which can be defined in a natural way).
The proof of the theorem can be easily adjusted to cover the case of
$\eps$-minimum points (with some fixed $\eps>0$).
See \cite[Theorem~6.7 and Corollary~6.10]{FabKruMeh24}, where the case $T=\{1,2\}$ was studied.

\item
The nonlocal \eqref{T6.1P11} and local \eqref{T6.1P12} primal necessary conditions
established in the proof of Theorem~\ref{T6.1} can be of interest.
They are valid in the setting of an arbitrary complete metric space (with distances in place of the norms).
Similar conditions are explicitly formulated in \cite[Theorem~6.6]{FabKruMeh24} for the case $T=\{1,2\}$.
\end{enumerate}
\end{remark}

Next, we
apply Theorem~\ref{T6.1}
in the context
of subdifferential calculus.
We consider a collection of functions $\{f_{t}\}_{t\in T}$ and their sum $\overline\sum_{t\in T}f_{t}$ defined by \eqref{f}.
The next statement presents a generalized version of the (strong) fuzzy sum rule for Fr\'{e}chet subdifferentials.

\begin{theorem}
\label{T7.1}
Let $X$ be a Banach space, $\bx\in\dom\overline\sum_{t\in T}f_{t}$, and $x^*\in{\sdf}\big(\overline\sum_{t\in T}f_{t}\big)(\bar x)$.
Suppose that
either $X$ is Asplund or $f_t$ $(t\in T)$ are convex, and
there is a number $\de>0$ such that
$f_t$ $(t\in T)$ are lower semicontinuous and bounded from below on $B_\de(\bx)$, and $\{f_{t}\}_{t\in T}$ is firmly
uniformly
lower semicontinuous on $B_\de(\bx)$.
Then,
for any
{$\eps>0$ and}
$S_0\in\mathcal{F}(T)$,
there exist an $S:=\{t_1,\ldots,t_m\}\in\mathcal{F}(T)$ with some $m\in\N$, and points
$x_1,\ldots,x_m\in B_\eps(\bx)$ such that $S_0\subset S$,
condition \eqref{T6.2-01} is satisfied, and
\begin{gather}
\label{SR-2}
x^*\in\sum_{i=1}^m{\sdf}f_{t_i}(x_i)+\varepsilon\B^*.
\end{gather}
\end{theorem}

\begin{proof}
Let
{$\eps>0$ and}
$t_0\notin T$.
Set $\eps':=\eps/(\|x^*\|+2)$ and  $f_{t_0}(x):=-\ang{x^*,x}+\eps'\|x-\bx\|$ for all $x\in X$.
For any $x\in X$, we have
$$\overline\sum_{t\in T\cup\{t_0\}}f_{t}(x) =\limsup_{S\uparrow(T\cup\{t_0\}),\;|S|<\infty}\, \sum_{t\in S}f_{t}(x)
=\limsup_{S\uparrow T,\;|S|<\infty}\,\sum_{t\in S} f_{t}(x)-\ang{x^*,x}+\eps'\|x-\bx\|.$$
By definition of the \Fr\ subdifferential,
\begin{multline*}
\liminf_{\bx\ne x\to\bx} \frac{
\big(\overline\sum_{t\in T\cup\{t_0\}}f_{t}\big)(x)- \big(\overline\sum_{t\in T\cup\{t_0\}}f_{t}\big)(\bx)} {\|x-\bx\|}
\\=\liminf_{\bx\ne x\to\bx} \frac{\big(\overline\sum_{t\in T}f_{t}\big)(x)-\big(\overline\sum_{t\in T}f_{t}\big)(\bx)-\ang{x^*,x-\bx}}{\|x-\bx\|} +\eps'\ge\eps'.
\end{multline*}
Hence, $\bx$ is a local minimum of $\overline\sum_{t\in T\cup\{t_0\}}f_{t}$.
Taking a smaller $\de$ if necessary, we can assume that $\bx$ is a minimum of $\overline\sum_{t\in T\cup\{t_0\}}f_{t}$ over $B_\de(\bx)$.
Since $f_{t_0}$ is uniformly continuous, it follows from
Corollary~\ref{C2.1} and Proposition~\ref{P3.2}
that $\{f_t\}_{t\in T\cup\{t_0\}}$ is uniformly
lower semicontinuous on $B_\de(\bx)$.
If $X$ is Asplund, then,
by the second part of Theorem~\ref{T6.1} there exist
an
$S:=\{t_1,\ldots,t_m\}\in\mathcal{F}(T)$ with some $m\in\N$ such that
$S_0\subset S$ and
$S\cup\{t_0\}\in\mathcal{F}(T\cup\{t_0\})$, and points $x_0,\ldots,x_m\in B_{\eps'}(\bx)\subset B_{\eps}(\bx)$
such that
\begin{gather}
\notag
\sum_{i=0}^m\big(f_{t_i}(x_i)-f_{t_i}(\bx)\big) <\ang{x^*,x_0
-\bx}-\eps'\|x_0-\bx\|+\eps' \le(\|x^*\|+1)\eps'<\eps.
\\
\label{T7.1P4}
0\in\sum_{i=0}^m{\sdf} f_{t_i}(x_i)+\eps'\B^*.
\end{gather}
If $f_t$ $(t\in T)$ are convex, the same conclusion follows from the first part of Theorem~\ref{T6.1} taking into account that in the convex case Clarke and \Fr\ subdifferentials coincide.
Observe that $\sd f_{t_0}(x_0)\subset-x^*+\eps'\overline\B^*$.
It follows from \eqref{T7.1P4} that
\begin{equation*}
d\Big(x^*,\sum_{i=1}^m{\sdf} f_{t_i}(x_i)\Big)<2\eps'\le\eps.
\end{equation*}
This completes the proof.
\qed
\end{proof}
\if{
\AK{11/08/23.
The general Banach space version of Theorem~\ref{T7.1} in terms of Clarke subdifferentials can hopefully be proved too, but this is going to require a little different machinery.

At the moment, I cannot drop the assumption in (ii) that $T$ is finite.
Unfortunately this diminishes the value of  quasiuniform lower semicontinuity in infinite settings. :-(
}
}\fi

\section{Quasiuniform lower semicontinuity}
\label{quasi}

Now we briefly discuss more subtle ``quasi'' versions of the uniform infimum and uniform lower semicontinuity properties studied in Section~\ref{S3}.
In this section, $X$ is a metric space
{and $U$ is a subset of $X$}.
The definitions below employ \emph{essentially interior} subsets of a given set.
Recall from \cite[Definition~2.1]{FabKruMeh24} that, {a subset $V\subset U$ is an essentially interior subset of $U$ if $B_\rho(V)\subset U$ for some $\rho>0$.
The latter condition is obviously equivalent to $\inf_{x\in V}d(x,X\setminus U)>0$.
The collection of all essentially interior subsets of $U$ is denoted by $EI(U)$.
Note that $V\in EI(X)$ for any $V\subset X$.}

The next simple lemma is going to be useful.

\begin{lemma}
\label{L2.1}
Let $g:X\to\R_{\infty}$.
{Then}
$\inf_{V\in EI(U)}\inf_Vg=\inf_{\Int U}g$.
As a consequence, if $U$ is open, then
$\inf_{V\in EI(U)}\inf_Vg=\inf_Ug$.
\end{lemma}

\begin{proof}
For any $V\in EI(U)$, we have $V\subset\Int U$.
Hence,
$\inf_{\Int U}g\le\inf_Vg$,
and consequently,
$\inf_{\Int U}g\le\inf_{V\in EI(U)}\inf_Vg$.
On the other hand, given any $x\in\Int U$, we have $\{x\}\in EI(U)$.
Hence,
$\inf_{V\in EI(U)}\inf_Vg\le g(x)$,
and consequently,
$\inf_{V\in EI(U)}\inf_Vg\le\inf_{\Int U}g$.
\qed
\end{proof}

The ``quasi'' analogues of definitions
\eqref{La0} and \eqref{Th0} look, respectively, as follows:
\begin{align}
\label{Ladag}
{\Lambda}_U^\dag(\{f_{t}\}_{t\in T}):=&
\inf_{V\in EI(U)}\liminf_{S\uparrow T,\;|S|<\infty}\;
\liminf\limits_{\substack{\diam\{x_t\}_{t\in S}\to0\\ x_t\in V\,(t\in S)}}\; \sum_{t\in S}f_{t}(x_{t}),
\\
\notag
{\Theta}_U^\dag(\{f_t\}_{t\in T}):=&\sup_{V\in EI(U)} \limsup_{S\uparrow T,\;|S|<\infty}\limsup_{\substack{\diam\{x_t\}_{t\in S}\to0\\ x_t\in(\dom f_t)\cap V\,(t\in S)}}
\\
\label{Thdag}
&\hspace{20mm}
\inf_{x\in U} \max\Big\{\max_{t\in S}d(x,x_t),\overline{\sum_{t\in T}}f_t(x)-\sum_{t\in S}f_{t}(x_{t})\Big\}.
\end{align}
Of course, the above quantities can only be useful if $\Int U\ne\es$ (which is the case in the situation of our main interest when $U$ is a \nbh\ of a given point).
By analogy with \eqref{La0-} and \eqref{La0}, we call \eqref{Ladag} the \emph{quasiuniform infimum}
of $\{f_{t}\}_{t\in T}$ over~$U$.

\begin{remark}
If $U=X$, the quantities ${\Lambda}_U^\dag(\{f_{t}\}_{t\in T})$ and
${\Theta}_U^\dag(\{f_t\}_{t\in T})$ coincide with ${\Lambda}(\{f_{t}\}_{t\in T})$
and ${\Theta}(\{f_t\}_{t\in T})$, respectively.
\end{remark}


The next definition gives ``quasi'' analogues of the properties in Definition~\ref{D3.2}.

\begin{definition}
\label{D3.14}
The collection $\{f_{t}\}_{t\in T}$ is
\begin{enumerate}
\item
\label{D3.14.1}
\emph{quasiuniformly lower semicontinuous} on $U$ if
$\inf_U\overline\sum_{t\in T}f_{t} \le{\Lambda}_U^\dag(\{f_{t}\}_{t\in T})$;
\item
\label{D3.14.2}
\emph{firmly quasiuniformly lower semicontinuous}
on $U$ if ${\Theta}_U^\dag(\{f_{t}\}_{t\in T})=0$.
\end{enumerate}
\end{definition}

It follows straightforwardly from comparing definitions
\eqref{La0}, \eqref{Th0}, \eqref{Ladag} and \eqref{Thdag} that
\begin{gather*}
{\Lambda}_U^\dag(\{f_{t}\}_{t\in T})\ge{\Lambda}_U(\{f_{t}\}_{t\in T}),\\
\inf_U\overline\sum_{t\in T}f_{t} -{\Lambda}_U^\dag(\{f_{t}\}_{t\in T}) \le{\Theta}_U^\dag(\{f_{t}\}_{t\in T}) \le{\Theta}_U(\{f_{t}\}_{t\in T});
\end{gather*}
hence, the next proposition holds true.

\begin{proposition}
\label{P3.15}
\begin{enumerate}
\item
\label{P3.15.1}
If $\{f_{t}\}_{t\in T}$ is uniformly (firmly uniformly) lower semicontinuous on $U$, then it is quasiuniformly (firmly quasiuniformly) lower semicontinuous on~$U$.
\item
\label{P3.15.2}
If $\{f_{t}\}_{t\in T}$ is firmly quasiuniformly lower semicontinuous on $U$, then it is quasiuniformly lower semicontinuous on $U$.
\end{enumerate}
\end{proposition}

\begin{remark}
\label{R3.17}
\begin{enumerate}
\item
By definitions \eqref{Ladag} and \eqref{f}, and Lemma~\ref{L2.1}, we have
\begin{gather}
\label{24}
{\Lambda}_U^{\dag}(\{f_{t}\}_{t\in T}) \le\inf_{V\in EI(U)}\liminf_{S\uparrow T,\;|S|<\infty}\inf_V\sum_{t\in S}f_{t} \le\inf_{V\in EI(U)}\inf_V\overline{\sum_{t\in T}}f_{t} =\inf_{\Int U}\overline{\sum_{t\in T}}f_{t},
\end{gather}
and consequently,
when $U$ is open,
the inequality in Definition~\ref{D3.14}\,\ref{D3.14.1} can only hold as equality in which case
the latter inequality can serve as a qualification condition.
\item
Employing \eqref{La0-} or \eqref{La0}, quantity \eqref{Ladag} can be represented as follows:
\begin{gather}
\label{R3.17-2}
{\Lambda}_{U}^\dag(\{f_{t}\}_{t\in T})
=\inf_{V\in EI(U)}{\Lambda}_{V}(\{f_t\}_{t\in T}) =\inf_{V\in EI(U)}\liminf_{S\uparrow T,\;|S|<\infty}\;
{\Lambda}_{V}(\{f_t\}_{t\in S}).
\end{gather}
Quantity \eqref{Thdag} does not in general allow similar recursive representations because the two sums in its \RHS\ are over different index sets and, additionally, the $\inf$ and the second $\limsup$ operations involve different subsets $U$ and $V$.
\item
\label{R3.17.3}
Denote
\begin{gather}
\label{Dedag}
{\Delta}_{U}^\dag(\{f_{t}\}_{t\in T}):=
\sup_{V\in EI(U)} \limsup_{S\uparrow T,\;|S|<\infty}\Big(\inf_{U} \sum_{t\in S} f_t- {\Lambda}_{V}\big(\{f_t\}_{t\in S}\big)\Big).
\end{gather}
Thanks to our standing conventions, the \RHS\ of \eqref{Dedag} is well defined.
Note that the infimum in \eqref{Dedag} is taken over $U$, while ${\Lambda}_{V}(\{f_t\}_{t\in S})$ involves a smaller set $V$.
In view of \eqref{f} and \eqref{R3.17-2}, we have
\begin{gather}
\label{R3.17-4}
{\Delta}_{U}^\dag(\{f_{t}\}_{t\in T}) \le\inf_{U}\overline\sum_{t\in T}f_t- {\Lambda}_{U}^\dag(\{f_t\}_{t\in T}),
\end{gather}
and consequently, the inequality in Definition~\ref{D3.14}\,\ref{D3.14.1} implies 
\begin{gather}
\label{R3.17-5}
{\Delta}_U^\dag(\{f_{t}\}_{t\in T})\le0.
\end{gather}
The latter condition defines a kind of \emph{weak} quasiuniform lower semicontinuity.
If $T$ is finite, inequality \eqref{R3.17-4} holds as equality, and the two properties coincide.

In view of \eqref{De0} and \eqref{Dedag}, we have
${\Delta}_{U}^\dag(\{f_{t}\}_{t\in T})\le{\Delta}_{U}(\{f_{t}\}_{t\in T})$,
and consequently, inequality
${\Delta}_{U}(\{f_{t}\}_{t\in T})\le0$ (cf. Remark~\ref{R3.2}\,\ref{R3.2.3}) implies
${\Delta}_{U}^\dag(\{f_{t}\}_{t\in T})\le0$.

\item
\label{R3.17.4}
In view of \eqref{Thdag} and \eqref{f}, we have
\begin{align}
\notag
{\Theta}_{U}^{\dag}(\{f_{t}\}_{t\in T})\ge
\sup_{V\in EI(U)} \limsup_{S\uparrow T,\;|S|<\infty} {\Theta}_{U,V}(\{f_{t}\}_{t\in S}),
\end{align}
where ${\Theta}_{U,V}(\{f_t\}_{t\in S})$ in \eqref{Thdag} is a two-parameter modification of \eqref{Th} depending on subsets $V\subset U\subset X$:
\begin{align}
\hspace{-2mm}
\label{ThV}
{\Theta}_{U,V}(\{f_t\}_{t\in S}) :=&\limsup_{\substack{\diam\{x_t\}_{t\in S}\to0\\ x_t\in\dom f_t\cap V\;(t\in S)}} \inf_{x\in U} \max\Big\{\max_{t\in S}d(x,x_t),\sum_{t\in S}(f_t(x)-f_{t}(x_{t}))\Big\}.
\hspace{-2mm}
\end{align}
\if{
\AH{09-08-2024. I suppose that here we also can remove the reference to the terms $\dom f_t$.}
\AK{12/08/24.
I am not sure.}
}\fi
Hence, the equality in Definition~\ref{D3.14}\,\ref{D3.14.2} implies
\begin{gather}
\label{R3.17-6}
\ds\sup_{V\in EI(U)} \limsup_{S\uparrow T,\;|S|<\infty} {\Theta}_{U,V}(\{f_{t}\}_{t\in S})=0.
\end{gather}
This condition defines a kind of \emph{weak} firm quasiuniform lower semicontinuity.
\end{enumerate}
\end{remark}

The following estimates are immediate consequences of representations \eqref{R3.17-2} and definition~\eqref{Dedag}:
\begin{align}
\label{28}
\liminf_{S\uparrow T,\;|S|<\infty}\inf_{U}\sum_{t\in S} f_t\le
{\Lambda}_{U}^\dag(\{f_t\}_{t\in T})+
{\Delta}_{U}^\dag(\{f_{t}\}_{t\in T})\le \limsup_{S\uparrow T,\;|S|<\infty}\inf_{U}\sum_{t\in S} f_t.
\end{align}
They yield necessary and sufficient conditions for the quasiuniform lower semicontinuity.

\begin{proposition}
\label{P3.18}
\begin{enumerate}
\item
\label{P3.18.1}
If ${\Delta}_{U}^\dag(\{f_{t}\}_{t\in T})\le0$ and
{$\overline\sum_{t\in T}f_{t}$ is $\inf$-stable on $U$}
in the sense of \eqref{P3.4-1}, then $\{f_{t}\}_{t\in T}$ is quasiuniformly lower semicontinuous on $U$.
\item
\label{P3.18.2}
If $U$ is open and $\{f_{t}\}_{t\in T}$ is quasiuniformly lower semicontinuous on $U$, then ${\Delta}_{U}^\dag(\{f_{t}\}_{t\in T})\le0$ and
{$\overline\sum_{t\in T}f_{t}$ is $\inf$-quasistable on $U$}
in the sense that
\begin{gather}
\label{P3.18-2}
\inf_U\overline{\sum_{t\in T}}f_{t}
\le\inf_{V\in EI(U)}\liminf_{S\uparrow T,\;|S|<\infty}\inf_V\sum_{t\in S}f_{t}.
\end{gather}
\end{enumerate}
\end{proposition}

\begin{proof}
\begin{enumerate}
\item
Let ${\Delta}_{U}^\dag(\{f_{t}\}_{t\in T})\le0$ and
condition \eqref{P3.4-1} be satisfied (with both infima over~$U$).
Using the first inequality in \eqref{28}, we obtain
\begin{gather*}
\inf_{U}\overline{\sum_{t\in T}} f_t- {\Lambda}_{U}^\dag(\{f_t\}_{t\in T})\le
\liminf_{S\uparrow T,\;|S|<\infty}\inf_U\sum_{t\in S}f_{t}- {\Lambda}_{U}^\dag(\{f_t\}_{t\in T})\le
{\Delta}_{U}^\dag(\{f_{t}\}_{t\in T})\le0;
\end{gather*}
hence, $\{f_{t}\}_{t\in T}$ is quasiuniformly lower semicontinuous on $U$.

\item
Let $U$ be open, and $\{f_{t}\}_{t\in T}$ be quasiuniformly lower semicontinuous on $U$.
It follows from the second inequality in \eqref{28} that
\begin{align*}
{\Delta}_{U}^\dag(\{f_{t}\}_{t\in T})\le \limsup_{S\uparrow T,\;|S|<\infty}\inf_{U}\sum_{t\in S} f_t-
{\Lambda}_{U}^\dag(\{f_t\}_{t\in T}) \le\inf_{U}\overline{\sum_{t\in T}} f_t-
{\Lambda}_{U}^\dag(\{f_t\}_{t\in T})\le0.
\end{align*}
The inequalities in \eqref{24} must hold as equalities.
This implies \eqref{P3.18-2}.
\qed\end{enumerate}
\end{proof}

\begin{remark}
The
{$\inf$-quasistability of the infinite sum}
property in Proposition~\ref{P3.18}\,\ref{P3.18.2} is a particular case of a more general property defined for an arbitrary family of functions $\varphi_{S}:X\to\R_{\infty}$ $(S\in\mathcal{F}(T))$ and the corresponding $\limsup$ function $\varphi_\infty$ given by~\eqref{tf}
by the inequality
\begin{align}
\label{R3.19-1}
\inf_U\varphi_\infty\le\inf_{V\in EI(U)} \liminf_{S\uparrow T,\;|S|<\infty}\inf_V \varphi_S.
\end{align}

Observe from \eqref{tf} and Lemma~\ref{L2.1} that
\begin{align*}
\liminf_{S\uparrow T,\;|S|<\infty}\inf\limits_U \varphi_S\le
\inf_{V\in EI(U)}\liminf_{S\uparrow T,\;|S|<\infty}\inf\limits_V \varphi_S\le&
\inf_{V\in EI(U)}\limsup_{S\uparrow T,\;|S|<\infty}\inf_V \varphi_S
\\\le&
\inf_{V\in EI(U)}\inf_{V}\varphi_\infty=
\inf_{\Int U} \varphi_\infty.
\end{align*}
Thanks to the first inequality,
{the $\inf$-stability}
property \eqref{R3.5-1} implies
the $\inf$-quasi\-stability
\eqref{R3.19-1} while, thanks to the last two inequalities, if $U$ is open,
inequality \eqref{R3.19-1} (hence, also \eqref{P3.18-2})
can only hold as equality, and we have
\begin{align*}
\inf\limits_U\varphi_\infty=
\inf\limits_{V\in EI(U)} \liminf_{S\uparrow T,\;|S|<\infty}\inf_V\varphi_S=
\inf\limits_{V\in EI(U)} \limsup_{S\uparrow T,\;|S|<\infty}\inf_V\varphi_S.
\end{align*}

If condition \eqref{R3.19-1} is satisfied and the infimum in the \LHS\ is attained at some point $\bx\in\Int U$, then $\{\bx\}\in EI(U)$ and we have
$\limsup_{S\uparrow T,\;|S|<\infty}\varphi_S(\bx) =\varphi_\infty(\bx)\le\liminf_{S\uparrow T,\;|S|<\infty}\varphi_S(\bx)$; hence,
$\varphi_\infty(\bx)=\lim_{S\uparrow T,\;|S|<\infty}\varphi_S(\bx)$.
\end{remark}

When checking condition ${\Theta}_{U}^\dag(\{f_{t}\}_{t\in T})=0$ in Definition~\ref{D3.14}\,\ref{D3.14.2}, one can drop the restriction $x\in U$ in definition \eqref{Thdag}.

\begin{proposition}
\label{P2.2}
The collection $\{f_{t}\}_{t\in T}$ is firmly quasiuniformly lower semicontinuous on $U$ if and only if
\begin{multline}
\label{P2.2-1}
\sup_{V\in EI(U)} \limsup_{S\uparrow T,\;|S|<\infty}\limsup_{\substack{\diam\{x_t\}_{t\in S}\to0\\ x_t\in(\dom f_t)\cap V\,(t\in S)}}\\
\inf_{x\in X} \max\Big\{\max_{t\in S}d(x,x_t),\overline{\sum_{t\in T}}f_t(x)-\sum_{t\in S}f_{t}(x_{t})\Big\}=0.
\end{multline}
When $T$ is finite, the last condition is equivalent to
\begin{gather*}
\sup_{V\in EI(U)} \limsup_{\substack{\diam\{x_t\}_{t\in T}\to0\\ x_t\in(\dom f_t)\cap V\,(t\in S)}}
\inf_{x\in X} \max\Big\{\max_{t\in T}d(x,x_t),{\sum_{t\in T}}(f_t(x)-f_{t}(x_{t}))\Big\}=0.
\end{gather*}
\end{proposition}

\begin{proof}
The ``only if'' part follows from comparing definition \eqref{Thdag} and the \LHS\ of \eqref{P2.2-1}.
Suppose that condition \eqref{P2.2-1} holds true.
Let $\eps>0$ and $V\in EI(U)$, i.e., $B_\rho(V)\subset U$ for some $\rho>0$.
Set $\eps':=\min\{\eps,\rho\}$.
By \eqref{P2.2-1}, for any $S_0\in\mathcal{F}(T)$, there exist an $S\in\mathcal{F}(T)$ and a $\de>0$ such that $S_0\subset S$ and, for any $x_{t}\in\dom f_t\cap V$ $(t\in S)$ with $\diam\{x_{t}\}_{t\in S}<\de$, one can find an $x\in X$ such that
\begin{align*}
\max_{t\in S}d(x,x_{t})<\eps'\le\eps
\AND
\overline{\sum_{t\in T}}f_t(x)-\sum_{t\in S} f_{t}(x_{t})<\eps'\le\eps.
\end{align*}
In particular, $d(x,x_{t})<\eps'\le\rho$ for all $t\in S$; hence, $x\in U$.
Thus,
\begin{gather*}
\inf_{x\in U} \max\Big\{\max_{t\in S} d(x,x_t),\overline{\sum_{t\in T}}f_t(x)-\sum_{t\in S} f_{t}(x_{t})\Big\}<\eps,
\end{gather*}
and it follows from \eqref{Thdag} that ${\Theta}_{U}^\dag(\{f_{t}\}_{t\in T})=0$, i.e., $\{f_{t}\}_{t\in T}$ is firmly quasiuniformly lower semicontinuous on $U$.
The last assertion is obvious.
\qed
\end{proof}

\begin{corollary}
\label{C3.5}
If $\{f_{t}\}_{t\in T}$ is firmly quasiuniformly lower semicontinuous on $U$, then it is firmly quasiuniformly lower semicontinuous on any subset of $U$.
\end{corollary}

\begin{proof}
The assertion follows from Proposition~\ref{P2.2}.
It suffices to observe that condition \eqref{P2.2-1} remains true if $U$ is replaced by a smaller set.
\qed
\end{proof}

Similar to its non-quasi analogue, the firm quasiuniform lower semicontinuity property in Definition~\ref{D3.14}\,\ref{D3.14.2} is stable under uniformly continuous perturbations of one of the functions.

\begin{proposition}
\label{P3.21}
Suppose that $\{f_{t}\}_{t\in T}$ is firmly quasiuniformly lower semicontinuous and $g\colon X\to\R$ is uniformly continuous on $U$.
Let $t_0\in T$.
Set $\tilde f_{t_0}:=f_{t_0}+g$ and $\tilde f_{t}:=f_{t}$ for all $t\in T\setminus\{t_0\}$.
Then the collection $\{\tilde f_t\}$ is firmly quasiuniformly lower semicontinuous on $U$.
\end{proposition}

The proof below is a slight modification of that of Proposition~\ref{P3.7}.

\begin{proof}
Let $\eps>0$ and $V\in EI(U)$.
There exists a $\rho>0$ such that $|g(x')-g(x'')|<\eps/2$ for all $x',x''\in U$ with $d(x',x'')<\rho$.
Set $\eps':=\min\{\eps/2,\rho\}$.
By Definition~\ref{D3.14}\,\ref{D3.14.2} and definition \eqref{Thdag}, for any $S_0\in\mathcal{F}(T)$, there exist an $S\in\mathcal{F}(T)$ with $S_0\cup\{t_0\}\subset S$ and a $\de>0$ such that, for any $x_{t}\in(\dom f_t)\cap V$
$(t\in S)$ with $\diam\{x_{t}\}_{t\in S}<\de$, one can find an $x\in U$ such that conditions \eqref{P3.7P1} are satisfied.
Then $d(x,x_{t_0})<\eps'\le\rho$.
Observe that $(\dom\tilde f_t)\cap U=(\dom f_t)\cap U$ for all $t\in S$, and conditions \eqref{P3.7P2} are satisfied.
Thus, $\{\tilde f_t\}$ is firmly quasiuniformly
lower semicontinuous on~$U$.
\qed
\end{proof}

The next statement is an immediate consequence of Proposition~\ref{P3.21} and the trivial fact that adding the identical zero to a collection of functions does not affect its firm quasiuniform lower semicontinuity property (cf. Proposition~\ref{P2.6}).
\begin{corollary}
\label{C3.22}
Suppose that $\{f_t\}_{t\in T}$ is firmly quasiuniformly
lower semicontinuous, $t_0\notin T$, $f_{t_0}\colon X\to\R$ is uniformly continuous on $U$.
Then the collection $\{f_t\}_{t\in T\cup\{t_0\}}$ is firmly quasiuniformly
lower semicontinuous on $U$.
\end{corollary}

\begin{remark}
\begin{enumerate}
\item
Facts similar to those in Propositions~\ref{P2.2} and \ref{P3.21} and Corollaries~\ref{C3.5} and \ref{C3.22} are also true for the weaker version of firm quasiuniform lower semicontinuity mentioned in Remark~\ref{R3.17}\,\ref{R3.17.4}.
\item
The assumption of firm quasiuniform lower semicontinuity plays an important role in the proofs of Propositions~\ref{P2.2} and \ref{P3.21}.
The analogues of
these propositions
and Corollaries~\ref{C3.5} and \ref{C3.22}
do not necessarily hold for the
non-firm
quasiuniform lower semicontinuity property.
\end{enumerate}
\end{remark}

The ``quasi'' analogues of Propositions~\ref{P4.6} and \ref{P4.7} exploit inf-compactness assumptions in order to guarantee
conditions \eqref{R3.17-5} and \eqref{R3.17-6}, which are essential for the quasiuniform and firm quasiuniform lower semicontinuity properties, respectively; see
Remark~\ref{R3.17}\,\ref{R3.17.3} and \ref{R3.17.4}.

\begin{proposition}
\label{P5.2}
Suppose that the functions $f_t$ $(t\in T)$ are
\lsc\ on $U$.
If there is a $t_0\in T$ such that the sets $\{x\in\cl V\mid f_{t_0}(x)\leq c\}$ are
compact for all $V\in EI(U)$ and $c\in\R$, and $\Lambda_{V}(\{f_t\}_{t\in S\setminus\{t_0\}})>-\infty$
for all $V\in EI(U)$ and $S\in\mathcal{F}(T)$ with $S_0\cup\{t_0\}\subset S$,
then condition \eqref{R3.17-5} is satisfied.
If, additionally,
\begin{align}
\label{P4.6-4}
\limsup_{\substack{\diam\{x_t\}_{t\in S}\to0\\ x_t\in\dom f_t\cap V\;(t\in S)}}\; \sum_{t\in S}f_{t}(x_t)<+\infty
\end{align}
for all $S\in\mathcal{F}(T)$ with $S_0\cup\{t_0\}\subset S$ and $V\in EI(U)$,
then
condition \eqref{R3.17-6} is satisfied.
\if{
\AH{Is {weakly} firmly quasiuniformly lower semicontinuous defined previously?}
\AK{12/06/25. Remark~\ref{R3.17}\,\ref{R3.17.4}.}
}\fi
\end{proposition}

\begin{proof}[Sketch of proof]
The proof goes along the same lines as that of Proposition~\ref{P4.6} with minor amendments identified below.
The sequences $\{x_{tk}\}$ need to be chosen from an arbitrary set $V\in EI(X)$.
As a result, a tail and the limiting points of the sequence $\{x_{t_0k}\}$ belong to the compact set $\{x\in \cl V\mid f_{t_0}(x)\leq c\}$.
This ensures
{condition \eqref{R3.17-5}.}
Using the weaker condition \eqref{P4.6-4} instead of \eqref{P4.6-3} ensures
{condition \eqref{R3.17-6}.}
\qed
\end{proof}

\begin{proposition}
\label{P5.6}
Let $X$ be a normed space,
and the functions $f_t$ $(t\in T)$ be weakly sequentially \lsc\ on $U$.
If $\cl^wV\subset U$ for each $V\in EI(U)$, and there is a exist a $t_0\in T$ and an $S_0\in\mathcal{F}(T)$ such that the sets $\{x\in\cl^wV\mid f_{t_0}(x)\leq c\}$ are weakly sequentially compact for all $V\in EI(U)$ and $c\in\R$, and $\Lambda_{V}(\{f_t\}_{t\in S\setminus\{t_0\}})>-\infty$
for all $V\in EI(U)$ and $S\in\mathcal{F}(T)$ with $S_0\cup\{t_0\}\subset S$,
{then
\begin{gather}
\label{P5.6-1}
\sup_{V\in EI(U)} \limsup_{S\uparrow T,\;|S|<\infty}\Big(\inf\sum_{t\in S} f_t- {\Lambda}_{V}\big(\{f_t\}_{t\in S}\big)\Big)\le0.
\end{gather}
}
\end{proposition}

\begin{corollary}
\label{C4.3}
Let $X$ be a reflexive Banach space, $U$ be convex and bounded, the functions $f_t$ $(t\in T)$ be weakly sequentially \lsc\ on~$U$,
and $\Lambda_{V}(\{f_t\}_{t\in S\setminus\{t_0\}})>-\infty$
for all $V\in EI(U)$ and $S\in\mathcal{F}(T)$ with $S_0\cup\{t_0\}\subset S$.
Then
condition \eqref{P5.6-1} is satisfied.
\end{corollary}

\begin{remark}
\begin{enumerate}
\item
As illustrated by \cite[Example~3.11]{FabKruMeh24}, condition \eqref{P4.6-4} in Proposition~\ref{P5.2} is essential.
\item
Condition $\cl^wV\subset U$
{in Proposition~\ref{P5.6}}
does not hold in general for $V\in EI(U)$, but it does in many interesting situations, for instance when $V$ is convex or
{a finite}
union of convex sets, or when $X$ is finite dimensional.
\item
In Proposition~\ref{P5.6}, if $X$ is Banach, then, by Eberlein–Šmulian theorem, it suffices to assume the level sets to be weakly  compact.
\end{enumerate}
\end{remark}

The key uniform (firm uniform) lower semicontinuity assumptions in Theorems~\ref{T6.1} and \ref{T7.1} can be weakened: it suffices to replace them by the corresponding quasiuniform (firm quasiuniform) lower semicontinuity assumptions.
This leads to slightly more advanced statements.

\begin{theorem}
\label{T6.01}
Theorem~\ref{T6.1} remains valid if the assumption
``$\{f_{t}\}_{t\in T}$ is uniformly lower semicontinuous on $B_\de(\bx)$''
is replaced by
``$\{f_{t}\}_{t\in T}$ is quasiuniformly lower semicontinuous on $B_\de(\bx)$''.
\end{theorem}

To prove Theorem~\ref{T6.01}, in view of Definition~\ref{D3.14}\,\ref{D3.14.1} and definition \eqref{Ladag}, one only needs to replace condition \eqref{T6.1P00} in the proof of Theorem~\ref{T6.1} by the following one:
\begin{gather*}
\overline{\sum_{t\in T}}f_{t}(\bx)\le\inf_{\rho\in(0,\de)}\; \liminf_{S\uparrow T,\;|S|<\infty}\; \liminf_{\substack{\diam\{x_t\}_{t\in S}\to0\\ x_t\in\overline B_\rho(\bar x)\;(t\in S)}}\;\sum_{t\in S}f_{t}(x_{t}).
\end{gather*}

\begin{theorem}
\label{T5.2}
Theorem~\ref{T7.1} remains valid if the assumption
``$\{f_{t}\}_{t\in T}$ is firmly uniformly lower semicontinuous on $B_\de(\bx)$''
is replaced by
``$\{f_{t}\}_{t\in T}$ is firmly quasiuniformly lower semicontinuous on $B_\de(\bx)$''.
\end{theorem}

The proof of Theorem~\ref{T5.2} repeats that of Theorem~\ref{T7.1} with the replacement of application of Theorem~\ref{T6.1} by that of Theorem~\ref{T6.01}.
\if{
Replacing application of Theorem~\ref{T6.1} by that of Theorem~\ref{T6.01} in the proof of Theorem~\ref{T7.1} leads to the following slightly more advanced statement.
{Its second part can be deduced from \cite[Theorem~7.1]{FabKruMeh24}.}

\begin{theorem}
Let $X$ be a Banach space, $\bx\in\dom\overline\sum_{t\in T}f_{t}$, and $x^*\in{\sdf}\big(\overline\sum_{t\in T}f_{t}\big)(\bar x)$.
\begin{enumerate}
\item
Suppose that
either $X$ is Asplund or $f_t$ $(t\in T)$ are convex, and
there is a number $\de>0$ such that
$f_t$ $(t\in T)$ are lower semicontinuous and bounded from below on $B_\de(\bx)$, and $\{f_{t}\}_{t\in T}$ is firmly quasiuniformly lower semicontinuous on $B_\de(\bx)$.
Then,
for any
{$\eps>0$ and}
$S_0\in\mathcal{F}(T)$,
there exist an $S\in\mathcal{F}(T)$ and points
$x_t\in B_\eps(\bx)$ $(t\in S)$ such that $S_0\subset S$, and
conditions \eqref{T6.2-01} and \eqref{SR-2} are satisfied.
\item
Suppose that $X$ is reflexive, $T$ is finite, and $f_t$ $(t\in T)$ are weakly sequentially lower semicontinuous.
Then,
{for any $\eps>0$,}
there exist points
$x_t\in B_\eps(\bx)$ $(t\in T)$ such that
$|f_{t}(x_t)-f_{t}(\bx)|<\eps$ $(t\in T)$, and condition \eqref{SR-2} is
satisfied with $S:=T$.
\end{enumerate}
\end{theorem}

\begin{proof}
The proof of the first assertion repeats that of Theorem~\ref{T7.1} with the replacement of application of Theorem~\ref{T6.1} by that of Theorem~\ref{T6.01}.
To prove the second assertion, it suffices to notice that, in view of Remark~\ref{R3.17}\,\ref{R3.17.3},
it follows from Corollary~\ref{C4.3}
that,
{for some $\de>0$},
$\{f_t\}_{t\in T\cup\{t_0\}}$ is
quasiuniformly lower semicontinuous on
$B_\de(\bx)$.
Without loss of generality, $f_t$ $(t\in T)$ are lower semicontinuous and bounded from below on $B_\de(\bx)$; and $\bx$ is a
minimum of $\overline\sum_{t\in T\cup\{t_0\}}f_{t}$ over $B_\de(\bx)$.
In view of Remark~\ref{R3.1}\,\ref{R3.1.4},
the conclusion follows from Theorem~\ref{T6.01}
with
{$S:=T$}
using the same argument as in the proof of Theorem~\ref{T7.1}.
\end{proof}
\fi

The next modification of Theorem~\ref{T5.2} with finite $T$ can be of interest.
\begin{proposition}
\label{T5.7}
Let $X$ be a reflexive Banach space, $f_1,\ldots,f_m$ be weakly sequentially lower semicontinuous,
$\bx\in\bigcap_{i=1}^m\dom f_{i}$, and $x^*\in{\sdf}\big(\sum_{i=1}^mf_{i}\big)(\bar x)$.
Then,
{for any $\eps>0$,}
there exist points
$x_1,\ldots,x_m\in B_\eps(\bx)$  such that
$|f_{i}(x_i)-f_{t}(\bx)|<\eps$ $(i=1,\ldots,m)$, and $x^*\in\sum_{i=1}^m{\sdf}f_{i}(x_i)+\varepsilon\B^*$.
\end{proposition}
\begin{proof}
It suffices to notice that, with $f_{0}(x):=-\ang{x^*,x}+\eps\|x-\bx\|$ for some $\eps>0$ and all $x\in X$, in view of Remark~\ref{R3.17}\,\ref{R3.17.3},
it follows from Corollary~\ref{C4.3}
that,
{for some $\de>0$},
$\{f_i\}_{i=0,\ldots,m}$ is
quasiuniformly lower semicontinuous on
$B_\de(\bx)$.
Without loss of generality, $f_i$ $(i=1,\ldots,m)$ are lower semicontinuous and bounded from below on $B_\de(\bx)$; and $\bx$ is a
minimum of $\sum_{i=1}^mf_{i}$ over $B_\de(\bx)$.
In view of Remark~\ref{R3.1}\,\ref{R3.1.4},
the conclusion follows from Theorem~\ref{T6.01}
with
{$S:=\{1,\ldots,m\}$}
using the same argument as in the proof of Theorem~\ref{T7.1}.
\end{proof}
\begin{remark}
Proposition~\ref{T5.7} can be deduced from \cite[Theorem~7.1]{FabKruMeh24}.
\end{remark}

\section{Conclusions}\label{sec:conclusions}

We extend the decoupling techniques developed in \cite{BorZhu96,BorIof96,Las01,FabKruMeh24} to arbitrary collections of functions.
Extensions of the notions of uniform lower semicontinuity and firm uniform lower semicontinuity
are discussed.
Several characterizations and sufficient conditions ensuring uniform lower semicontinuity and firm
uniform lower semicontinuity are provided, thus, extending the decoupling theory to infinite
collections of functions. We show that the firm uniform lower semicontinuity property is
stable under uniformly continuous perturbations of one of the functions.

The main Theorem~\ref{T6.1} gives fuzzy subdifferential necessary conditions (multiplier rules) for a
local minimum of the sum of an arbitrary collection of functions.
As a consequence, we establish a generalized version of the (strong) fuzzy sum rule for Fr\'{e}chet subdifferentials of an arbitrary collection of functions without the traditional Lipschitz continuity assumptions.

More subtle ``quasi'' versions of the uniform infimum and uniform lower semicontinuity properties, which are still sufficient for the
multiplier rules and lead to a slightly more advanced fuzzy sum rule for Fr\'echet
subdifferentials, are discussed.

In 
\cite{HanKruLop2}, we consider a standard application of the mentioned results to optimality of an infinite
sum of functions coming from semi(infinite) optimization, where there are
infinitely many constraints.
In the context of the duality theory developed
in \cite{Duf56,Kre61,Bor83} (see also
\cite{GobVol22,DinGobLopVol23}) we have infinitely many associated dual
variables, and the Lagrangian function contains infinite sums of functions.
We apply the 
fuzzy multiplier
methodology
proposed in this paper
for closing the duality gap and
guaranteeing strong duality in convex optimization with an infinite
number of constraints. One possibility consists of considering the
infinitely many dual variables (associated with the infinitely many
constraints) but taking the zero-value except for a finite number of them;
this is so-called \emph{Haar dual} (see, e.g., \cite[Chapters~2 and 8]{GobLop98},
and the references therein). But, from the mathematical point of view and also having in mind
possible applications, it can make sense to enlarge the dual space to
include spaces like $\ell _{\infty }$ and $\ell _{1}$. Two different
associated dual problems are studied in \cite{HanKruLop2}.
By applying
minimax-type theorems we 
show there that Slater condition guarantees
strong duality in both cases, improving some well-known results in
semi-infinite linear optimization. Moreover, the assumption of uniform lower
semicontinuity of the functions involved in the (optimal) Lagrangian allows
us to deduce fuzzy primal optimality conditions. It turns out that this
methodology presents remarkable advantages with respect to alternative
approaches as the one based on the \emph{Farkas-Minkowski property} (see, e.g., \cite{GobLop98}), and the addition of consequent inequalities to the primal set of
constraints (see, e.g., \cite[Theorem 8.2.12 and Corollary 8.2.13]{CorHanLop23}).

\section*{Acknowledgements}
\small
The authors have benefited from fruitful discussions with Pedro P\'erez-Aros  and wish to express him their gratitude.

This work was supported by Grant PID2022-136399NB-C21 funded by MICIU/AEI/10.13039/501100011033 and by ERDF/EU, MICIU of Spain and Universidad de Alicante (Contract Beatriz Galindo BEA-GAL 18/00205), and by AICO/2021/165 of Generalitat Valenciana, and by
Programa Propio Universidad de Alicante (INVB24-01 - Ayudas para Estancias de Personal Investigador Invitado 2025).

Parts of the work were done during Alexander Kruger's stays at the University of Alicante and the Vietnam Institute for Advanced Study in Mathematics in Hanoi.
He is grateful to both institutions for their hospitality and supportive environment.

\section*{Declarations}

\noindent{\bf Data availability. }
Data sharing is not applicable to this article as no datasets have been generated or analysed during the current study.

\noindent{\bf Conflict of interest.} The authors have no competing interests to declare that are relevant to the content of this article.

\addcontentsline{toc}{section}{References}

\bibliographystyle{spmpsci}
\bibliography{buch-kr,kruger,kr-tmp}

\end{document}